\documentclass[final,leqno]{siamltex}

\usepackage{amsmath}
\usepackage{amssymb}
\usepackage{graphicx}
\usepackage{color}
\usepackage{ifpdf}
\usepackage{url}

\def\norm#1{\|#1\|}

\newcommand{\Eproof}{\hfill$\square$}

\newtheorem{algorithm}{\normalfont\textsc{Algorithm}}
\newtheorem{assumption}{Assumption A.\!\!}

\newcommand{\dnorm}[1]{|\!\|#1\|\!|}
\newcommand{\abs}[1]{\left\vert#1\right\vert}
\newcommand{\dom}{\mathrm{dom}}
\renewcommand{\int}{\mathrm{int}}
\newcommand{\argmin}{\mathrm{arg}\!\min}
\newtheorem{remark}{Remark}
\newcommand{\aref}[1]{\textrm{\textsc{A.\ref{#1}}}}

\title{An Inexact Perturbed Path-Following Method for Lagrangian Decomposition in Large-Scale Separable Convex Optimization}

\author{Quoc Tran Dinh$^{*\dagger}$, Ion Necoara$^{\ddagger}$, Carlo Savorgnan$^{*}$ \and
Moritz Diehl\thanks{Department of Electrical Engineering (ESAT-SCD) and
Optimization in Engineering Center (OPTEC), K.U. Leuven, Kasteelpark Arenberg
10, B-3001 Leuven, Belgium ({\tt \{quoc.trandinh, carlo.savorgnan, moritz.diehl\}@esat.kuleuven.be}).
\newline{ $^{\dagger}$ Department of Mathematics-Mechanics-Informatics, Hanoi University of Science, Hanoi, Vietnam.}
\newline{ $^{\ddagger}$ Automation and Systems Engineering Department, University Politehnica Bucharest, 060042 Bucharest, Romania;
(\texttt{ion.necoara@upb.ro})} } }

\begin{document}
\maketitle

\begin{abstract}
In this paper, we propose an inexact perturbed  path-following algorithm in the framework of Lagrangian dual decomposition for solving large-scale structured
convex optimization  problems. Unlike the exact versions considered in literature, we allow one to solve the primal problem inexactly up to a given accuracy.
The inexact perturbed algorithm allows to use both approximate Hessian matrices and approximate gradient vectors to  compute Newton-type directions for the dual
problem. The algorithm is divided into two phases. The first phase computes an initial point which makes use of inexact perturbed damped Newton-type iterations,
while the second one performs the path-following algorithm with inexact perturbed full-step Newton-type iterations. We analyze the convergence of both phases
and estimate the worst-case complexity.
As a special case, an exact path-following algorithm for Lagrangian relaxation is derived and its worst-case complexity is estimated.
This variant possesses some differences compared to the previously known methods. Implementation details are discussed and numerical results are reported.
\end{abstract}


\pagestyle{myheadings}
\thispagestyle{plain}
\markboth{Q. Tran Dinh, I. Necoara, C. Savorgnan and M. Diehl}{An Inexact Perturbed Path-Following Method for Lagrangian Decomposition}

\section{Introduction}\label{sec:intro}
Many optimization problems arising in networked systems, image processing, data mining,  economics, distributed model predictive control and multi-stage
stochastic optimization can be formulated as a separable convex optimization problem, see, e.g.
\cite{Conejo2006,Komodakis2010,Naveen2007,Necoara2009,Samar2007,Venkat2008,Xiao2004}.
If the optimization problem has moderate size or possesses sparsity structure, then it can  be solved  efficiently by standard optimization methods. In many
practical situations, we can encounter problems which may not be easy to solve by standard optimization algorithms due to the high dimensionality or the
distributed locations of the data and devices. However, many problems can be reformulated as separable convex optimization problems such that the subproblems
generated from their components can be solved in a closed form or more easier than the full problem.

In this paper, we are interested in the following convex separable optimization problem:
\begin{equation}\label{eq:s1_main_CP}
\left\{\begin{array}{cl}
\displaystyle\max_{x\in\mathbb{R}^{n}} &\Big\{ \phi(x) := \displaystyle\sum_{i=1}^M\phi_i(x_{i})\Big\} \\
\textrm{s.t.}~ &x_{i} \in X_{i},~ (i=1,\cdots, M),\\
& \displaystyle\sum_{i=1}^MA_{i}x_{i} = b,
\end{array}\right.
\end{equation}
where $x = (x_{1}^T, \dots, x_{M}^T)^T$ with $x_{i}\in\mathbb{R}^{n_i}$ is a vector of  decision variables, $\phi_i:\mathbb{R}^{n_i} \to \mathbb{R}$ is concave,
$X_{i}$ is a nonempty, closed convex subset in $\mathbb{R}^{n_i}$, $A_{i}\in\mathbb{R}^{m\times n_i}$, $b\in\mathbb{R}^m$ for all $i=1,\dots, M$, and $n_1+n_2
+\cdots + n_M = n$. The last constraint is usually referred to as a \textit{coupling linear constraint}.
Problems of the form \eqref{eq:s1_main_CP} were considered in many research papers, see, e.g. \cite{Bertsekas1989,Necoara2008,Necoara2009,Quoc2011c}. Note that
coupling linear inequality constraints of the form $\sum_{i=1}^MB_{i}x_{i} \leq d$ can also be formulated into \eqref{eq:s1_main_CP} by using slack variables,
see, e.g. \cite{Necoara2009}.

Several methods solve problem \eqref{eq:s1_main_CP} by decomposing it into small subproblems that can be solved separately by standard
optimization techniques. For instance, by applying Lagrangian relaxation, the coupling constraint can be brought into the objective function and, by the
separability, we can decompose the dual function into small subproblems \cite{Bertsekas1989}. However, using such a Lagrangian relaxation technique generally
leads to a nonsmooth optimization problem. There are several attempts to overcome this difficulty by smoothing the dual function. One can add an augmented
Lagrangian term or a proximal term to the objective function of the problem. Unfortunately, the first approach breaks the separability of the original problem
due to the cross terms between the components. Therefore, the second approach is more suitable for this type of problems.

Recently,  smoothing techniques in convex optimization have attracted increasing interest and found many applications \cite{Nesterov2005a}. In the framework of
the Lagrangian dual decomposition, there are two popular approaches. The first approach is regularization. By adding a regularization term as a proximal
term to the objective function, the primal subproblem becomes strongly convex. Consequently, the master dual problem is smooth which allows one to apply
smoothing optimization techniques \cite{Bertsekas2010,Chen1994,Necoara2008,Quoc2011c}. The second approach is using barrier functions, this technique is
suitable for problems with conic constraints \cite{Fukuda2002,Kojima1993,Mehrotra2009,Necoara2009,Shida2008,Zhao1999,Zhao2001,Zhao2005}.
Several methods in this direction are based on the fact that, by using a self-concordant log-barrier function, the family of the dual functions which depend on
a barrier parameter is strongly self-concordant in the sense of Nesterov and Nemirovski \cite{Nesterov1994} under certain assumptions. Consequently,
path-following methods can be used to solve the master dual problem.
Note that this technique is only applicable to the cases where either the objective function is linear, quadratic and self-concordant or the problem is
compatible in the sense that it possesses a property that makes the smooth objective function of the dual self-concordant. Several methods in this direction
require a crucial assumption that the primal subproblems are solved exactly.
In practice, solving exactly the primal subproblems to compute the dual function is only conceptual. Any numerical optimization method
provides an approximate solution and, consequently, the dual function is also approximated. This paper studies an inexact perturbed path-following method in the
framework of Lagrangian decomposition for solving \eqref{eq:s1_main_CP}.

\paragraph{Contribution}
The contribution of this paper is fivefold.
\begin{enumerate}
\item By applying smoothing technique via self-concordant barrier functions, we provide a local and a global smooth approximation to the dual function and
estimate the approximation error.
\item A new inexact perturbed path-following decomposition algorithm is proposed for solving \eqref{eq:s1_main_CP}. The algorithm consists of two phases. Both
phases allow the primal subproblems to be solved approximately. Moreover, the algorithm is highly parallelizable.
\item The convergence theory is investigated under standard assumptions used in any interior point method and the worst-case complexity is estimated.
\item When the primal problem is assumed to be solved exactly, our method reduces to the path-following method for Lagrangian decomposition considered in
\cite{Mehrotra2009,Necoara2009,Shida2008,Zhao2005}. However, the variants presented in this papers possesses a larger neighborhood of the analytic center where
convergence is guaranteed.
\item The implementation details are discussed and numerical experiments are implemented to confirm the theoretical development.
\end{enumerate}
Let us emphasize some difference between the method presented in this paper and the previously known methods.
\begin{enumerate}
\item Even though smoothing techniques based on self-concordant barriers are not new, in this paper we do not only apply smoothing techniques to the dual
problem but also provide some properties of the smooth function.
The smooth approximation of the dual function only requires that the objective function is convex (not necessarily smooth).
However, the dual function is smooth, which allows us to use any smooth optimization technique such as gradient-based methods or sequential quadratic
programming-based (SQP) methods to solve the master problem.

\item The new algorithm allows us to solve the primal subproblems inexactly where we can control the accuracy up to $\delta^{*}\approx 0.043286$ (see Section
\ref{sec:inexact_perturbed_path_following_alg} for more details) such that at the early steps of the path-following algorithm, they can be solved very
inexactly. This point is significant if the primal subproblems require high computational cost.
Note that the algorithm developed in this paper is different from the one considered in \cite{Zhao1999} for linear programming, where the inexactness of the
primal subproblems is defined in a different way.

\item Based on a recent monograph \cite{Nesterov2004}, we directly analyze the convergence of the algorithm. This makes our theory self-contained. Moreover, it
also allows us to optimally choose the parameters and to trade-off between the convergence rate of the master problem and the accuracy of the primal
subproblems.

\item In the exact case, the variant in this paper still has some advantages compared with the previous ones. Firstly, the radius of the neighborhood of the
analytic center is $(3-\sqrt{5})/2 \approx 0.38197$ which is larger than $2-\sqrt{3} \approx 0.26795$ of previous methods. Secondly, since the performance of
an interior point algorithm crucially depends on the parameters of the algorithm, we analyze directly the path-following iteration to select these parameters in
an optimal way.
\end{enumerate}

The rest of this paper is organized as follows. In the next section, we briefly describe the Lagrangian dual decomposition method applied to separable convex
optimization. Section \ref{sec:smoothing_technique} deals with a smoothing technique for the dual function via self-concordant barriers and investigates the
main properties of the smooth dual function. Section \ref{sec:inexact_perturbed_path_following_alg} presents an inexact perturbed path-following decomposition
algorithm. The convergence of the algorithm is analyzed and the worst-case complexity is estimated. Section \ref{sec:path_following_for_exact_case} considers
an exact variant of the algorithm presented in Section \ref{sec:inexact_perturbed_path_following_alg}. Section \ref{sec:implementation_detail} discusses
implementation details of the algorithms. Section \ref{sec:num_results} shows numerical tests and a comparison. Concluding remarks are included in the last
section.  The proofs of the technical statements are given in the appendix.

\paragraph{Notation and Terminology}
Throughout the paper, we  shall work on the Euclidean space $\mathbb{R}^n$ endowed with an inner product $x^Ty$ for $x, y\in\mathbb{R}^n$ and the Euclidian norm
$\norm{x} = \sqrt{x^Tx}$. The notation $x = (x_1, \dots, x_M)$ defines a vector in $\mathbb{R}^n$ formed from $M$ sub-vectors $x_i\in\mathbb{R}^{n_i}$,
$i=1,\dots, M$, where $n_1+ \cdots + n_M = n$.

For a proper, lower semi-continuous  convex function $f$, the notation $\textrm{dom}(f)$ denotes the domain of $f$, $\overline{\textrm{dom}}(f)$ is the closure
of $\textrm{dom}(f)$ and $\partial f(x)$ denotes the subdifferential of $f$ at $x$. For a concave function $f$ we also denote by $\partial{f(x)}$ as the
``super-differential'' of $f$ at $x$, where $\partial{f(x)}  := -\partial\{-f(x)\}$ .
Let $f$ be twice continuously differentiable and convex on $\mathbb{R}^n$.
For a given vector $u$, the local norm of $u$ with respect to $f$ at $x$, where $\nabla^2f(x)$ is positive definite, is defined as $\norm{u}_x :=
\left[u^T\nabla^2f(x)u\right]^{1/2}$ and its dual norm is $\norm{u}_x^{*} := \max\{u^Tv ~|~ \norm{v}_x \leq 1 \} = \left[u^T\nabla^2f(x)^{-1}u\right]^{1/2}$.
Clearly, $\abs{u^Tv} \leq \norm{u}_x\norm{v}_x^{*}$. Let $F$ be a standard self-concordant function, $W^0(x,r) := \{z\in\mathbb{R}^n ~|~ \norm{z-x}_x < r\}$
defines the \textit{Dikin} ellipsoid of $F$ at $x$, where $\norm{z-x}_x = [(z-x)^T\nabla^2F(x)(z-x)]^{1/2}$.

For a given symmetric matrix $P$ in $\mathbb{R}^{n\times n}$, the expression $P\succeq 0$ (resp. $P\succ 0$) means  that $P$ is positive semi-definite (resp.
positive definite); $P \succeq Q$ and $P\preceq Q$ (resp. $P\succ Q$ and $P\prec Q$) mean that $P-Q$ and $Q-P$ are positive semidefinite (resp. positive
definite), respectively.

The notation $\mathbb{R}_{+}$ and $\mathbb{R}_{++}$ define the set of non-negative and positive numbers, respectively. The function $\omega:\mathbb{R}_{+}\to
\mathbb{R}$ is  defined by $\omega(t) := t - \ln(1+t)$ and its dual  $\omega_{*} : [0, 1]\to\mathbb{R}$ is defined by $\omega_{*}(t) := -t - \ln(1-t)$. Note
that both functions are convex, nonnegative and increasing. For a real number $x$, $\lfloor{x}\rfloor$ denotes the largest integer number which is less than or
equal to $x$.

\section{Lagrangian dual relaxation in convex optimization}\label{sec:lagrangian_decomposition}
A classical technique to address coupling constraints in separable convex optimization is based on  Lagrangian relaxation \cite{Bertsekas1989}. We briefly
review such a technique in this section.

Without loss of generality we consider problem \eqref{eq:s1_main_CP} with $M=2$. The separable convex optimization problem \eqref{eq:s1_main_CP}, with $M=2$,
can be expressed as:
\begin{equation}\label{eq:s2_sep_CP2}
\phi^{*} := \left\{\begin{array}{cl}
\displaystyle\max_{x:=(x_1, x_2)} \!\!\!\! & \left\{ \phi(x):=\phi_1(x_{1}) + \phi_2(x_{2}) \right\}\\
\textrm{s.t.} &A_{1}x_{1} + A_{2}x_{2} = b,\\
&x \in X := X_{1}\times X_{2}.
\end{array}\right.
\end{equation}
Let us define $A := [A_{1}, A_{2}]$ and $n := n_1 + n_2$. The linear coupling constraint $A_{1}x_{1} + A_{2}x_{2} = b$ can be written as $Ax = b$. The Lagrange
function for problem \eqref{eq:s2_sep_CP2} with respect to the coupling constraint $A_{1}x_{1} + A_{2}x_{2} = b$ is defined as:
\begin{equation}\label{eq:s2_Lagrange_funtion}
L(x, y) := \phi(x) + y^T(Ax-b) = \phi_1(x_{1}) + \phi_2(x_{2}) + y^T(A_{1}x_{1} + A_{2}x_{2} - b), \nonumber
\end{equation}
where $y \in\mathbb{R}^m$ is the Lagrange multiplier associated with the coupling constraint. A pair $(x^{*}_0, y^{*}_0) \in X \times \mathbb{R}^m$ is called a
saddle point of $L$ if
\begin{equation}\label{eq:saddle_point}
L(x, y^{*}_0) \leq L(x^{*}_0, y^{*}_0) \leq L(x^{*}_0, y), ~ \forall x\in X, ~\forall y\in\mathbb{R}^m. \nonumber
\end{equation}
The dual problem of \eqref{eq:s2_sep_CP2} is
\begin{equation}\label{eq:s2_original_dual_prob}
d_0^{*} := \min_{y\in\mathbb{R}^m} d_0(y),
\end{equation}
where $d_0$ is the dual function which is defined as
\begin{equation}\label{eq:s2_original_dual_func}
d_0(y) := \max_{x\in X}\left\{ \phi_1(x_{1}) + \phi_2(x_{2}) + y^T(A_{1}x_{1}+A_{2}x_{2} - b) \right\}.
\end{equation}
If \textit{strong duality} holds at $(x^{*}_0, y^{*}_0)$ with $x^{*}_0 := (x_{0,1}^{*}, x_{0,2}^{*})\in X$ and $y^{*}_0\in\mathbb{R}^m$,  then we have
\cite{Boyd2004}:
\begin{equation}\label{eq:s2_minmax_prob}
d^{*}_0 = d_0(y^{*}_0)  = \min_{y\in\mathbb{R}^m}d_0(y) = \max_{x\in X}\left\{ \phi(x) ~|~ Ax = b\right\} = \phi(x^{*}_0) = \phi^{*}. \nonumber
\end{equation}
Let us denote by $X^{*}$ the solution set of \eqref{eq:s2_sep_CP2} and by $Y^{*}$ the solution set of the dual  problem \eqref{eq:s2_original_dual_prob}. It is
well-known that if either the Slater condition holds, i.e. $\textrm{ri}(X)\cap \{x\in\mathbb{R}^n ~|~ Ax = b\} \neq \emptyset$, where $\textrm{ri}(X)$ is the
relative interior of the convex set $X$, or $X$ is polyhedral, then $Y^{*}$ is bounded \cite{Boyd2004}.

Finally, it is important to notice that the dual function $d_0(\cdot)$ can be computed separately  by 
\begin{eqnarray} \label{eq:s2_represented_orig_dual_func}
&&d_0(y) = d_{0,1}(y) + d_{0,2}(y) - b^Ty, \nonumber\\
[-1ex]
\mathrm{where}~~ && \\
[-1ex]
&&  d_{0,i}(y) := \displaystyle\max_{x_{i}\in X_{i}}\left\{\phi_i(x_{i}) + y^TA_{i}x_{i} \right\},~ i=1, 2. \nonumber
\end{eqnarray}
Let $x_{0,i}^{*}(y)$ be a solution of the maximization problem in \eqref{eq:s2_original_dual_func} ($i=1,2$),  and $x^{*}_0(y) := (x_{0,1}^{*}(y),
x_{0,2}^{*}(y))$. Lagrangian relaxation generally leads to a nonsmooth optimization problem in the dual form. Consequently, numerical solution to the dual
problem encounters many drawbacks.

\section{Smoothing technique via self-concordant barriers}\label{sec:smoothing_technique}
Let us assume that the feasible set $X_{i}$ is convex, has nonempty interiors and possesses a $\nu_i$-self-concordant barrier $F_i$ for $i=1,2$. Theory
of self-concordant functions and self-concordant barriers can be found in \cite{Hertog1992,Nesterov1994,Nesterov2004}.
Throughout the paper, we use the following assumptions.

\begin{assumption}\label{as:A1}
\begin{itemize}
\item[]$\mathrm{(a)}$ The solution set $X^{*}$ of \eqref{eq:s2_sep_CP2} is nonempty. Either the Slater condition for \eqref{eq:s2_sep_CP2} is satisfied or $X$
is polyhedral.
\item[]$\mathrm{(b)}$ The feasible set $X_{i}$ is bounded in $\mathbb{R}^{n_i}$ with $\int(X_i)\neq\emptyset$ and possesses a self-concordant barrier $F_i$ with
parameter $\nu_i$ for $i=1,2$.
\item[]$\mathrm{(c)}$ The function $\phi_i$ is proper, upper semicontinuous and concave on $X_i$ for $i=1,2$.
\item[]$\mathrm{(d)}$ The matrix $A$ is \textit{full-row rank}.
\end{itemize}
\end{assumption}
Note that Assumptions \aref{as:A1}.a) and \aref{as:A1}.c) are standard in convex optimization, which guarantee the solvability of the problem and strong
duality.
Assumption \aref{as:A1}.b) can be satisfied by assuming that the set of the sample points generated by such an optimization algorithm is bounded.
Assumption \aref{as:A1}.d) is not restrictive since it can be satisfied by applying standard linear algebra techniques to eliminate redundant constraints.

\begin{remark}\label{re:s3_remark_on_assumption_A1}
As we can see in Section \ref{sec:implementation_detail}, the convex feasible set $X_{i}$ can be given as follows
\begin{equation}\label{eq:X_i}
X_{i} := X^c_{i} \cap X_{i}^a, ~~ X_{i}^a := \left\{x_{i}\in\mathbb{R}^{n_i} ~|~ E_{i}x_{i} = f_{i}\right\}, \nonumber
\end{equation}
where $\int(X_{i}^c)$ is nonempty and $X_i^c$ possesses a $\nu_i$-self-concordant barrier $F_i$.
Let $E = [E_{1}, E_{2}]$ be a matrix formed from $E_{i}$ and $A/E$ be a reduced form of $\begin{bmatrix}A\\ E\end{bmatrix}$ and $\int(X_{i}) :=
\int(X_{i}) \cap X_{i}^a$ for $i=1,2$.
In this case, the theory developed in the next sections can be extended to the problem with this constraint, see, e.g. \cite{Necoara2009}.
\end{remark}

Let us denote by $x_{i}^c$ the \textit{analytic center} of $X_{i}$,
which is defined as:
\begin{equation*}\label{eq:s3_analytic_center}
x_{i}^c := \textrm{arg}\!\!\!\!\!\!\min_{x_{i}\in \textrm{ri}(X_{i})}F_i(x_{i}), ~i=1,2.
\end{equation*}
Under Assumption A.\ref{as:A1}.b), $x^c := (x^c_{1}, x^c_{2})$ is well-defined due to \cite[Corollary 2.3.6]{Renegar2001}.
To compute $x^c$, one can apply the algorithms proposed in \cite[pp. 204--205]{Nesterov2004}.
Moreover, the following estimates hold:
\begin{eqnarray}\label{eq:s3_analytic_center_est}
F_i(x_{i}) - F_i(x_{i}^c) \geq \omega(\norm{x_{i}-x_{i}^c}_{x_{i}^c}) ~\mathrm{and} ~\norm{x_{i} - x_{i}^c}_{x_{i}^c} \leq \nu_i + 2\sqrt{\nu_i},
\end{eqnarray}
for all $x_{i} \in \overline{\mathrm{dom}}(F_i)$ and $i=1,2$ \cite[Theorems 4.1.13 and 4.2.6]{Nesterov2004}.

\subsection{A smooth approximation of the dual function}\label{subsec:bounds_on_dual_function}
Similarly to \cite{Kojima1993,Necoara2009,Shida2008,Zhao2005}, we construct a smooth approximation of the  nonsmooth dual function $d_0$ defined by
\eqref{eq:s2_original_dual_func} via self-concordant barriers.

Let us define the following functions:
\begin{eqnarray}\label{eq:s3_smoothed_dual_func}
&&d_i(y, t) := \!\!\! \max_{x_{i}\in\textrm{int}(X_{i})} \!\left\{ \phi_i(x_{i}) \!+\! y^TA_{i}x_{i} \!-\! t[F_i(x_{i}) \!-\! F_i(x^c_{i})]\right\}, ~i=1,2, \nonumber\\
[-1ex]
\textrm{and}~~ &&\\
[-1ex]
&& d(y, t) :=  d_1(y, t) + d_2(y, t) - b^Ty, \nonumber
\end{eqnarray}
where $t>0$ is referred to as a smoothness or barrier parameter. Note that, due to the strict concavity of the objective function, the maximization problem in
\eqref{eq:s3_smoothed_dual_func} has a unique solution, which is denoted by $x_{i}^{*}(y,t)$. Consequently, the functions $d_i(\cdot,t)$ ($i=1,2$) and
$d(\cdot,t)$ are well-defined and smooth on $\mathbb{R}^m$ for any $t > 0$. As in \cite{Zhao2005} we refer to $d$ as a \textit{smooth dual approximation}  of
$d_0$ and to the maximization problem in \eqref{eq:s3_smoothed_dual_func} as a \textit{primal subproblem}.

If we denote by $x^{*}(y,t) := (x^{*}_{1}(y,t), x^{*}_{2}(y,t))$, then we can write
\begin{equation*}
d(y,t) = \phi(x^{*}(y,t)) + y^T(Ax^{*}(y,t) - b) - t[F(x^{*}(y,t)) - F(x^c)].
\end{equation*}
The optimality condition for \eqref{eq:s3_smoothed_dual_func} is \begin{equation}\label{eq:s3_optimality_for_smoothed_dual_prob_i}
0 \in \partial\phi_i(x_{i}^{*}(y,t)) + A_{i}^Ty - t\nabla F_i(x_{i}^{*}(y,t)), ~i=1,2,
\end{equation}
where $\partial\phi_i(x_{i}^{*}(y,t))$ is the super-differential of $\phi_i$ at $x_{i}^{*}(y,t)$ ($i = 1,2$).
Since problem \eqref{eq:s3_smoothed_dual_func} is convex, this condition is necessary and sufficient.

Associated with the smooth dual function $d(\cdot, t)$, we consider the following \textit{master problem}:
\begin{eqnarray}\label{eq:s3_smoothed_dual_prob}
d^{*}(t) := \min_{y\in Y} d(y, t).
\end{eqnarray}
We denote by $y^{*}(t)$ a solution of \eqref{eq:s3_smoothed_dual_prob} if it exists and by $x^{*}(t) := x^{*}(y^{*}(t), t)$.

For a given $\beta \in (0, 1)$, we define a neighbourhood in $\mathbb{R}^m$ with respect to $F_i$ and $t > 0$ as
\begin{equation*}
\mathcal{N}^{F_i}_t(\beta) := \left\{ y\in \mathbb{R}^m ~|~ \lambda_{F_i}(x^{*}_{i}(y,t)) := \norm{\nabla F_i(x^{*}_{i}(y,t))}_{x^{*}_{i}(y,t)}^{*} \leq \beta \right\}.
\end{equation*}

The following lemma provides a local estimate for $d_0(\cdot)$, whose proof can be found in the appendix.
\begin{lemma}\label{le:s3_local_bound_on_dual_func}
Under Assumption A.\ref{as:A1} and $\beta\in (0, 1)$,  the function
$d(\cdot, t)$ defined by \eqref{eq:s3_smoothed_dual_func} satisfies:
\begin{eqnarray}\label{eq:s3_local_bound_on_dual_func}
0 \!\leq\! t\!\left[\!\sum_{i=1}^2\omega\!\!\left(\norm{x^{*}_{i}(y,t) \!-\! x^c_{i}}_{x^c_{i}}\!\right) \!\right] \!\leq\! d_0(y) \!-\! d(y,t) \!\leq\!
t\sum_{i=1}^2\!\left[\omega_{*}(\lambda_{F_i}(x^{*}_{i}(y,t))) \!+\! \nu_i\right],
\end{eqnarray}
for all $y\in \mathcal{N}^{F_1}_t(\beta)\cap \mathcal{N}^{F_2}_t(\beta)$.
\end{lemma}

From Lemma \ref{le:s3_local_bound_on_dual_func}, we see that
\begin{equation*}
0 \leq d_0(y) - d(y,t) \leq t[2\omega_{*}(\beta) + \nu_1+\nu_2], ~~\forall y \in \mathcal{N}^{F_1}_t(\beta)\cap \mathcal{N}^{F_2}_t(\beta).
\end{equation*}
Hence, for $t = t_f > 0$ sufficiently small, $d(\cdot, t_f)$ is a local approximation to $d_0(\cdot)$.

Under Assumption \aref{as:A1}, the dual optimal solution set $Y^{*}$ is bounded.
Without loss of generality, we can assume that $Y$ is bounded such that $Y^{*}\subset Y$.
Let
\begin{equation*}\label{eq:s3_dc_definition}
d^c(y) := \phi(x^c) + y^T(Ax^c - b) ,
\end{equation*}
where $x^c$ is the analytic center of $X$. From \eqref{eq:s2_original_dual_func} we have:
\begin{eqnarray}\label{eq:s3_diff_of_orig_dual_func_and_dc}
d_0(y) - d^c(y) &&= \max_{x\in X}\left\{\phi(x) + y^T(Ax-b)\right\} - \left[\phi(x^c) + y^T(Ax^c-b)\right] \geq 0, ~\forall y\in Y. \nonumber
\end{eqnarray}
Furthermore,
\begin{align}\label{eq:s3_upper_bound_of_dy_minus_dc}
0 \leq d_0(y) - d^c(y) &= \max_{x\in X}\{\phi(x) - \phi(x^c) + y^TA(x-x^c)\}\nonumber \\
& \overset{\tiny\phi~\mathrm{is~concave}}{\leq} \sum_{i=1}^2\max_{x_{i}\in X_{i}}\left\{\max_{\xi_{i}\in\partial\phi_i(x^c_{i})}\left\{\left[\xi_{i} + A_{i}^Ty\right]^T(x_{i} \!-\!
x^c_{i})\right\}\right\} \nonumber\\
&\leq \sum_{i=1}^2\max_{x_{i}\in X_{i}}\left\{\max_{\xi_{i}\in\partial\phi_i(x^c_{i})}\left\{\norm{\xi_{i} + A_{i}^Ty}_{x^c_{i}}^{*}\norm{x_{i} \!-\!
x^c_{i}}_{x^c_{i}}\right\}\right\} \\
& \overset{\tiny\eqref{eq:s3_analytic_center_est}}{\leq} \sum_{i=1}^2(\nu_i + 2\sqrt{\nu_i})\max_{\xi_{i}\in\partial\phi_i(x^c_{i})}\left\{ \norm{\xi_{i} + A_{i}^Ty}_{x^c_{i}}^{*}
\right\} \nonumber\\
& \leq K_1 + K_2 < +\infty, ~\forall y\in Y, \nonumber
\end{align}
where $K_i := (\nu_i + 2\sqrt{\nu_i})\max_{\xi_{i}\in\partial\phi_i(x^c_{i})}\big\{ \norm{\xi_{i} + A_{i}^Ty}_{x^c_{i}}^{*}
\big\}$ ($i=1,2$).
The following lemma shows that $d(\cdot, t)$ is a global approximation to $d_0(\cdot)$. The proof can be found in the appendix.

\begin{lemma}\label{le:s3_bound_of_smoothed_dual_func_to_orig_dual_func}
Suppose that Assumption A.\ref{as:A1} is satisfied. Then, for any $t > 0$ and $y\in Y$, the following estimate holds:
\begin{eqnarray}\label{eq:s3_bound_of_smoothed_dual_func_to_orig_dual_func}
0 \!\leq\! t\!\sum_{i=1}^2\!\omega(\norm{x^{*}_{i}(y,t) \!-\! x^c_{i}}_{x^c_{i}}) \!\leq\! d_0(y) \!-\! d(y,t) \leq t[\bar{\zeta}(K_1;\nu_1, t) \!+\! \bar{\zeta}(K_2;\nu_2, t)],
\end{eqnarray}
where $\bar{\zeta}(\tau; a, b) := a\left(1 + \max\left\{0, \ln\left(\frac{\tau}{ab}\right)\right\}\right)$ and $K_1$ and $K_2$ are two constants given in
\eqref{eq:s3_upper_bound_of_dy_minus_dc}.
\end{lemma}

The proof of the following statement can also be found in the appendix.
\begin{lemma}\label{le:s3_choice_of_para_t}
For a given tolerance $\varepsilon_d > 0$, if we choose $t > 0$ such that
\begin{equation}\label{eq:s3_choice_of_t}
0 \!<\! t \!\leq\! \bar{t} \!:=\! \min\!\left\{ \! \frac{K_1}{\nu_1}\kappa^{1/\kappa}\!,~ \frac{K_2}{\nu_2}\kappa^{1/\kappa}\!,~
\varepsilon_d^{1/(1-\kappa)}\left(\sum_{i=1}^2\nu_i +
(K_i/\nu_i)^{\kappa}\right)^{-1/(1-\kappa)}\!\!\right\},
\end{equation}
for fixed $\kappa\in (0,1)$, then it follows from Lemma \ref{le:s3_bound_of_smoothed_dual_func_to_orig_dual_func} that
\begin{equation*}
d(y, t) \leq d_0(y) \leq d(y, t) + \varepsilon_d.
\end{equation*}
In other words, if we fix $t_f\in (0, \bar{t})$ and  minimize $d(\cdot, t_f)$ over $Y$, then $y^{*}(t_f)$ is an $\varepsilon_d$-solution of
\eqref{eq:s2_original_dual_prob}.
\end{lemma}

Since $d(\cdot, t)$ is continuously differentiable, smooth optimization techniques such as gradient-based or SQP-based methods can be applied to solve problem
\eqref{eq:s3_smoothed_dual_prob}. If we choose $t_f > 0$ sufficiently small, then according to Lemmas \ref{le:s3_local_bound_on_dual_func} and
\ref{le:s3_bound_of_smoothed_dual_func_to_orig_dual_func}, we can obtain an approximate solution of \eqref{eq:s2_original_dual_prob} with a desired accuracy.

\subsection{The self-concordance of the smooth dual function}\label{subsec:selfconcordancy}
If the function $-\phi_i$ is self-concordant on $\textrm{dom}(-\phi_i)$ with parameter $M_{\phi_i}$, then the family of the functions $\phi_i(\cdot, t) :=
tF(\cdot)-\phi_i(\cdot)$ is also self-concordant on $\dom(-\phi_i)\cap\dom(F_i)$. Consequently, the smooth dual function $d(\cdot, t)$ is self-concordant as
stated in the following lemma. The proof of this lemma can be found, for instance, in  \cite{Mehrotra2009,Necoara2009,Shida2008,Zhao2005}.

\begin{lemma}\label{le:s3_self_concordancy}
Suppose that Assumption A.\ref{as:A1} is satisfied. Suppose further that $-\phi_i$ is $M_{\phi_i}$-self-concordant. Then, the function $d_i(\cdot, t)$ defined
by \eqref{eq:s3_smoothed_dual_func} is self-concordant with parameter $M_{d_i} := \max\{M_{\phi_i}, \frac{2}{\sqrt{t}}\}$ for any $t>0$ and $i=1,2$.
Consequently, $d(\cdot, t)$ is self-concordant with parameter $M_{d} = \max\{M_{\phi_1}, M_{\phi_2}, \frac{2}{\sqrt{t}}\}$.
\end{lemma}

Similar to standard path-following methods \cite{Nesterov1994,Nesterov2004}, in the following discussion, we assume that $\phi_i$ is linear as stated in
Assumption A.\ref{as:A2} below.

\begin{assumption}\label{as:A2}
The function $\phi_i$ is linear, i.e. $\phi_i(x_{i}) := c^T_{i}x_{i}$ for $i=1,2$.
\end{assumption}
Let $c := (c_{1}, c_{2})$ be the vector formed from $c_{i}$ ($i=1,2$).
Assumption \aref{as:A2} implies that $tF - \phi$ is $\frac{2}{\sqrt{t}}$-self-concordant. According to Lemma \ref{le:s3_self_concordancy}, $d_i(\cdot, t)$ is
$\frac{2}{\sqrt{t}}$-self-concordant.
Since $\phi_i$  is linear, if we denote by  $F(x) := F_1(x_{1}) + F_2(x_{2})$ the self-concordant barrier of $X$ with the parameter $\nu := \nu_1 + \nu_2$, then
the optimality condition \eqref{eq:s3_optimality_for_smoothed_dual_prob} is reduced to
\begin{eqnarray}\label{eq:s3_optimality_for_smoothed_dual_prob}
c + A^Ty - t\nabla F(x^{*}(y, t)) = 0.
\end{eqnarray}
The following lemma provides an explicit formula for the derivatives of $d(\cdot, t)$. The proof can be found in \cite{Necoara2009,Zhao2005}.

\begin{lemma}\label{le:s3_derivs_of_sm_dual_func}
Suppose that Assumptions \aref{as:A1} and \aref{as:A2} are satisfied. Then the first and second order derivatives of $d(\cdot,  t)$ on $Y$ are respectively
given as
\begin{eqnarray}\label{eq:s3_deriv_of_sm_dual_func}
&&\nabla d(y, t) = Ax^{*}(y, t) - b ~~\mathrm{and} ~~\nabla^2d(y, t) = \frac{1}{t}A\nabla^2F(x^{*}(y, t))^{-1}A^T,
\end{eqnarray}
where $x^{*}(y, t) = (x^{*}_{1}(y, t), x^{*}_{2}(y, t))$ is the solution of the primal subproblem in \eqref{eq:s3_smoothed_dual_func}.
\end{lemma}

Note that since $A$ is full-row rank and $\nabla^2F(x^{*}(y, t))$ is positive definite, matrix $\nabla^2{d}(y, t)$ is nonsingular for any $y\in Y$. Moreover,
since $F(x)$ and $\phi$ are separable, the Hessian matrix $\nabla^2F$ is block diagonal and they can also be evaluated \textit{in parallel}, see Section
\ref{sec:implementation_detail} for more details about implementation issues.

Now, since $d(\cdot, t)$ is $\frac{2}{\sqrt{t}}$ self-concordant, if we define
\begin{equation}\label{eq:s3_standard_sm_dual_func}
\tilde{d}(y, t) := \frac{1}{t}d(y, t),
\end{equation}
then $\tilde{d}(\cdot, t)$ is standard self-concordant, i.e. $M_{\tilde{d}} = 2$, due to \cite[Corollary 4.1.2]{Nesterov2004}.
For a given vector $v\in\mathbb{R}^m$, we define the norm $\norm{v}_y$ with respect to $\tilde{d}(\cdot,t)$ as $\norm{v}_y :=
[v^T\nabla^2\tilde{d}(y,t)v]^{1/2}$.

\subsection{Recovering the optimality and the feasibility}
It remains to show the relations between the master problem \eqref{eq:s3_smoothed_dual_prob}, the dual problem \eqref{eq:s2_original_dual_prob} and the
original primal problem \eqref{eq:s2_sep_CP2}. We first prove the following lemma.

\begin{lemma}\label{le:s3_properties_of_smoothed_dual_func}
Let Assumption A.\ref{as:A1} be satisfied. Then:
\begin{itemize}
\item[$\mathrm{a)}$] $d(y, \cdot)$ is non-increasing in $\mathbb{R}_{++}$ for a given $y\in Y$.
\item[$\mathrm{b)}$] $d^{*}(\cdot)$ defined by \eqref{eq:s3_smoothed_dual_prob} is differentiable and non-increasing in $\mathbb{R}_{++}$.
\item[$\mathrm{c)}$] It holds that $d^{*}(t) \leq d^{*}_0$ and $\displaystyle\lim_{t\downarrow 0^{+}}d^{*}(t) = d^{*}_0=\phi^*$.
 Moreover, $x^{*}(t)$ is feasible for problem \eqref{eq:s2_sep_CP2}.
\end{itemize}
\end{lemma}

\begin{proof}
Since the function $\xi(x, y, t) := \phi(x) + y^T(Ax-b) - t[F(x) - F(x^c)]$ is strictly concave and linear on $t$, it is well-known that $d(y, t) =
\displaystyle\max_{x\in \int(X)}\xi(x, y, t)$ is differentiable with respect to $t$ and its derivative is given by $\eta'(t) = -[F(x^{*}(y,t)) - F(x^c)]\leq
-\omega(\norm{x^{*}(y,t) - x^c}_{x^c}) \leq 0$ by \eqref{eq:s3_analytic_center_est}. Thus $d(y, \cdot)$ is nonincreasing in $t$ which proves a).

Now, we prove b) and c). From the definitions of $d^{*}(\cdot)$, $d(y,\cdot)$ and $y^{*}(\cdot)$ in \eqref{eq:s3_smoothed_dual_prob}, by using strong duality,
we have
\begin{align}\label{eq:s3_proof2_est1}
d^{*}(t) &= d(y^{*}(t), t) = \min_{y\in Y} d(y, t) \nonumber\\
& = \min_{y\in Y}\max_{x\in\int(X)}\left\{\phi(x) + y^T(Ax-b) - t[F(x)-F(x^c)]\right\} \nonumber\\
& = \max_{x\in\int(X)}\min_{y\in Y}\left\{\phi(x) + y^T(Ax-b) - t[F(x)-F(x^c)]\right\} \\
& = \max_{x\in\int(X)}\left\{\phi(x) - t[F(x)-F(x^c)] ~|~ Ax = b\right\} \nonumber \\
& = \phi(x^{*}(t)) - t[F(x^{*}(t)) - F(x^c)]. \nonumber
\end{align}
It follows from the forth line of \eqref{eq:s3_proof2_est1} that $d^{*}(\cdot)$ is differentiable and nonincreasing in $\mathbb{R}_{++}$.
Moreover, since $x^c$ is the analytic center of $X$, we have $F(x^{*}(t)) - F(x^c) \geq \omega(\norm{x^{*}(t) - x^c}_{x^c})$ due to
\eqref{eq:s3_analytic_center_est}. This inequality implies that $d^{*}(t) \leq \phi(x^{*}(t)) \leq \phi^{*} = d^{*}_0$.
On the other hand, from the forth line of \eqref{eq:s3_proof2_est1}, we also deduce that $x^{*}(t)$ is feasible to \eqref{eq:s2_sep_CP2}.
Furthermore, since $d^{*}(\cdot)$ is continuous on $\mathbb{R}_{++}$, we have $\lim_{t\downarrow0^{+}}d^{*}(t) = d^{*}_0$ which proves c).
\end{proof}

Let us define the Newton decrement of $\tilde{d}(\cdot,t)$ as follows:
\begin{eqnarray}\label{eq:s3_sm_dual_Newton_decrement}
\lambda = \lambda_{\tilde{d}(\cdot, t)}(y) := \norm{\nabla\tilde{d}(y,t)}_{y}^{*} =
\left[\nabla\tilde{d}(y,t)\nabla^2\tilde{d}(y,t)^{-1}\nabla\tilde{d}(y,t)\right]^{1/2}.
\end{eqnarray}
The following lemma shows the gap between $d(\cdot, t)$ and $d^{*}(t)$.

\begin{lemma}\label{le:s3_dual_gap_estimate}
Suppose that Assumption A.\ref{as:A1} is satisfied. Then, for any $y\in Y$ and $t > 0$ such that $\lambda_{\tilde{d}(\cdot,t)}(y) < 1$, one has
\begin{equation}\label{eq:s3_dual_gap_estimate_1}
0 \leq t\omega(\lambda_{\tilde{d}(\cdot,t)}(y)) \leq d(y,t) - d^{*}(t) \leq t\omega_{*}(\lambda_{\tilde{d}(\cdot,t)}(y)).
\end{equation}
Consequently, it holds that
\begin{equation}\label{eq:s3_dual_gap_estimate_2}
d(y,t) - d^{*}_0 = d(y,t) - \phi^{*} \leq t\omega_{*}(\lambda_{\tilde{d}(\cdot,t)}(y)).
\end{equation}
\end{lemma}

\begin{proof}
Since $\tilde{d}(\cdot, t)$ is standard self-concordant, for any $y\in Y$ such that $\lambda_{\tilde{d}(\cdot,t)}(y) < 1$, and $y^{*}(t) =
\displaystyle\argmin_{y\in Y}\tilde{d}(y,t)$,
by applying \cite[Theorem 4.1.13, inequality 4.1.17]{Nesterov2004}, we have
\begin{eqnarray*}
0\leq \omega(\lambda_{\tilde{d}(\cdot,t)}(y)) \leq \tilde{d}(y,t) - \tilde{d}(y^{*}(t),t) \leq \omega_{*}(\lambda_{\tilde{d}(\cdot,t)}(y)).
\end{eqnarray*}
This inequality is indeed \eqref{eq:s3_dual_gap_estimate_1} due to \eqref{eq:s3_standard_sm_dual_func}.
To prove \eqref{eq:s3_dual_gap_estimate_2}, we note that $d^{*}(t) - d^{*}_0 \leq 0$ by Lemma \ref{le:s3_derivs_of_sm_dual_func} c), adding this inequality to
\eqref{eq:s3_dual_gap_estimate_1} and noting that $d^{*}_0 = \phi^{*}$ we obtain \eqref{eq:s3_dual_gap_estimate_2}.
\end{proof}

We can also estimate a lower bound for $d^{*}(t) - d_0^{*}$. Since $F$ is convex, by using \eqref{eq:s3_analytic_center_est}, we have
\begin{equation*}
F(x) - F(x^c) \leq \nabla{F}(x)^T(x-x^c) \leq \norm{\nabla{F}(x)}^{*}_{x^c}\norm{x-x^c}_{x^c} \leq (\nu+2\sqrt{\nu})\norm{\nabla{F}(x)}^{*}_{x^c}.
\end{equation*}
Since $X$ is bounded and $\nabla{F}$ is continuous, using the above inequality, we have $c_X^F := \max_{x\in X}\norm{\nabla{F}(x)}^{*}_{x^c} < +\infty$. Thus
it follows from the last inequality that $\max_{x\in X}\{F(x) - F(x^c)\} \leq (\nu+2\sqrt{\nu})c^F_X < +\infty$.
Moreover, for any functions $u, v$ on $Z$, we have $\displaystyle\max_{z\in Z}\{u(z) - v(z)\} \geq \max_{z\in Z}u(z) - \max_{z\in Z} v(z)$. Finally, we
estimate $d^{*}(t) - d_0^{*}$ as
\begin{align*}
d^{*}(t) &= \min_{y\in Y}d(y,t) = \min_{y\in Y}\left\{\max_{x\in\int(X)}\left\{L(x,y) - t[F(x) - F(x^c)]\right\}\right\} \nonumber\\
& \geq \min_{y\in Y}\left\{\max_{x\in\int(X)}\left\{L(x,y)\right\} - t\max_{x\in\int(X)}\left\{F(x) - F(x^c)\right\}\right\} \nonumber\\
& \geq \min_{y\in Y}\max_{x\in X}L(x,y) - t\max_{x\in\int(X)}\left\{F(x) - F(x^c)\right\} \nonumber\\
& \geq d^{*}_0 - t(\nu+2\sqrt{\nu})c^F_X.
\end{align*}
Combining this inequality with \eqref{eq:s3_dual_gap_estimate_1} we obtain
\begin{equation*}
d(y,t) - d^{*}_0 \geq t\left[\omega(\lambda_{\tilde{d}(\cdot,t)}(y)) - (\nu+2\sqrt{\nu})c^F_X\right].
\end{equation*}

Now, we define an approximate solution of the dual problem
\eqref{eq:s2_original_dual_prob} as follows:
\begin{definition}\label{de:s3_epsilon_dual_sol}
For a given tolerance $\varepsilon_d > 0$, a point $y^{*}(t)$ is said to be an $\varepsilon_d$-solution of \eqref{eq:s2_original_dual_prob} if $0 \leq d^{*}_0
- d^{*}(t) \leq \varepsilon_d$.
\end{definition}

Let $y^{*}(t)$ be an $\varepsilon_d$-solution of \eqref{eq:s2_original_dual_prob} and $y \in Y$ such that $\lambda=\lambda_{\tilde{d}(\cdot,t)}(y)\leq\beta$ for
a fixed $\beta \in (0, 1)$. We have
\begin{equation*}
0 \leq d^{*}_ 0 - d(y,t)\leq \abs{d(y,t)-d^{*}(t)} + \abs{d^{*}(t)-d^{*}_0} \leq \varepsilon_d + t\omega_{*}(\lambda_{\tilde{d}(\cdot,t)}(y)) \leq \varepsilon_d
+ \omega_{*}(\beta)t.
\end{equation*}
Consequently, if we choose $t$ such that $t \leq \omega_{*}(\beta)^{-1}\varepsilon_d$ then
\begin{equation}\label{eq:s3_eps_dual_sol}
0 \leq d^{*}_0 - d(y, t)  = \phi^{*} - d(y,t) \leq 2\varepsilon_d.
\end{equation}
The algorithms presented in the next sections aim to find  a $2\varepsilon_d$-approximate solution of the dual problem \eqref{eq:s2_original_dual_prob} in the
sense of \eqref{eq:s3_eps_dual_sol}. Thus $d(y,t)$ is a $2\varepsilon_d$-approximation of the optimal value $\phi^{*}$.

It remains to quantify the feasibility gap of the original problem \eqref{eq:s2_sep_CP2} with respect to the coupling equality constraint $Ax = b$. We define
this \textit{feasibility gap} with respect to $x^{*}(y,t)$ as follows:
\begin{equation}\label{eq:s3_feasibility}
\mathcal{G}_{\mathrm{feas}}(y,t) := \norm{Ax^{*}(y,t) - b}_y^{*}.
\end{equation}
Here, $x^{*}(y,t) \in \int(X)$. From \eqref{eq:s3_feasibility}, \eqref{eq:s3_standard_sm_dual_func} and \eqref{eq:s3_sm_dual_Newton_decrement} and noting that
$\lambda \leq \beta$, we have:
\begin{equation*}
\mathcal{G}_{\mathrm{feas}}(y,t) = \norm{\nabla{d}(y,t)}_{y}^{*} = t\lambda \leq t\beta.
\end{equation*}
Therefore, with $t \leq \omega_{*}(\beta)^{-1}\varepsilon_d$ the feasibility gap reaches:
\begin{equation*}
\mathcal{G}_{\mathrm{feas}}(y,t) \leq \beta\omega_{*}(\beta)^{-1}\varepsilon_d.
\end{equation*}

\section{Inexact perturbed path-following method for Lagrangian decomposition}\label{sec:inexact_perturbed_path_following_alg}
This section presents an inexact perturbed path-following algorithm for solving approximately \eqref{eq:s2_original_dual_prob}.

\subsection{Inexact solution of the primal subproblem}\label{subsec:s4_inexactness_def}
Firstly, we define an inexact solution of \eqref{eq:s3_smoothed_dual_func} by using local norms. For given $y\in Y$ and $t > 0$, suppose that we allow to solve
approximately \eqref{eq:s3_smoothed_dual_func} up to a given accuracy $\bar{\delta} \geq 0$. More precisely, we define this approximate
solution as follows:
\begin{definition}\label{de:s4_app_sol}
A vector $\bar{x}_{\bar{\delta}}(y,t)$ is said to be a $\bar{\delta}$-approximate solution of $x^{*}(y,t)$ if
\begin{equation}\label{eq:s4_approx_sol}
\norm{\bar{x}_{\bar{\delta}}(y,t) - x^{*}(y,t)}_{x^{*}(y,t)} \leq \bar{\delta}.
\end{equation}
\end{definition}
Associated with $\bar{x}_{\bar{\delta}}(\cdot)$, we define the following function:
\begin{eqnarray}\label{eq:s4_inex_sm_dual_func}
d_{\bar{\delta}}(y, t) := c^T\bar{x}_{\bar{\delta}}(y,t) + y^T(A\bar{x}_{\bar{\delta}}(y,t) - b) - t[F(\bar{x}_{\bar{\delta}}(y,t)) - F(x^c)].
\end{eqnarray}
This function can be considered as an \textit{inexact smooth dual version} of $d_0$.
Next, we introduce two quantities:
\begin{equation}\label{eq:s4_derivs_of_inex_sm_dual_func}
\nabla d_{\bar{\delta}}(y, t) := A\bar{x}_{\bar{\delta}}(y,t) - b, ~~\mathrm{and}~~ \nabla^2d_{\bar{\delta}}(y,t) :=
\frac{1}{t}A\nabla^2F(\bar{x}_{\bar{\delta}}(y,t))^{-1}A^T.
\end{equation}
Since $x^{*}(y,t)\in\dom(F)=\int(X)$, we can choose an appropriate $\bar{\delta}\geq 0$ such that $\bar{x}_{\bar{\delta}}(y,t)\in\dom(F)$. Hence,
$\nabla^2F(\bar{x}_{\bar{\delta}}(y,t))$ is positive definite which means that $\nabla^2d_{\bar{\delta}}$ is well-defined.
Note that $\nabla d_{\bar{\delta}}$ and $\nabla^2d_{\bar{\delta}}$ are not the gradient vector and Hessian matrix of $d_{\bar{\delta}}(\cdot, t)$. However, due
to Lemma \ref{le:s3_derivs_of_sm_dual_func} and \eqref{de:s4_app_sol}, we can consider these quantities as an approximate gradient vector and Hessian matrix of
$d(\cdot, t)$, respectively.

Let
\begin{equation}\label{eq:s4_standard_inx_sm_dual_func}
\tilde{d}_{\bar{\delta}}(y,t) := \frac{1}{t}d_{\bar{\delta}}(y,t),
\end{equation}
and $\bar{\lambda}$ be the inexact Newton decrement of $\tilde{d}_{\delta}$ which is defined by
\begin{eqnarray}\label{eq:s4_inex_dual_Newton_decrement}
\bar{\lambda} = \bar{\lambda}_{\tilde{d}_{\bar{\delta}}(\cdot, t)}(y) := \dnorm{\nabla\tilde{d}_{\bar{\delta}}(y,t)}_y^{*}
= \left[\nabla\tilde{d}_{\bar{\delta}}(y,t)\nabla^2\tilde{d}_{\bar{\delta}}(y,t)^{-1}\nabla\tilde{d}_{\bar{\delta}}(y,t)\right]^{1/2}.
\end{eqnarray}
Here, we use the norm $\dnorm{\cdot}_y$ to distinguish it from $\norm{\cdot}_y$.

\subsection{The algorithmic framework}\label{subsec:s4_alg_framework}
From Lemma \ref{le:s3_dual_gap_estimate} we see that if we can generate a sequence $\{(y^k, t_k)\}_{k\geq 0}$ such that $\lambda_k := \lambda_{\tilde{d}(\cdot,
t_k)}(y^k) \leq \beta < 1$, then
\begin{equation*}
d(y^k,t_k) \uparrow d^{*}_0 = \phi^{*} ~ \mathrm{and}~ \mathcal{G}_{\mathrm{feas}}(y^k,t_k) \to 0, ~\mathrm{as}~ t_k\downarrow 0^{+}.
\end{equation*}
The aim of the algorithm is to generate $\{(y^k, t_k)\}_{k\geq 0}$ such that $\lambda_k \leq \beta < 1$.
First, we fix $t = t_0 > 0$ and find a point $y^0\in Y$ such that $\lambda_{\tilde{d}(\cdot, t_0)}(y^0) \leq \beta$. Then we simultaneously update $y$ and
$t$ such that $t$ tends to zero. The algorithmic framework is presented as follows.

\textsc{Inexact-Perturbed Path-following algorithmic framework.}
\begin{itemize}
\item[]\textbf{Initialization. } Choose an appropriate $\beta\in (0, 1)$ and a tolerance $\varepsilon_d > 0$. Fix $t=t_0 > 0$.
\item[]\textbf{Phase 1: }(\textit{Determine a starting point $y^0\in Y$ such that $\lambda_{\tilde{d}(\cdot, t_0)}(y^0) \leq \beta$}).
\begin{itemize}
\item[] Choose an initial vector $y^{0,0} \in Y$. Set $j = 0$.
\item[] \texttt{For $j=0,1, \dots$ perform}
\begin{enumerate}
\item If $\lambda_j := \lambda_{\tilde{d}(\cdot, t_0)}(y^{0,j}) \leq \beta$ then set $y^0 := y^{0,j}$ and terminate.
\item Solve the primal subproblems \eqref{eq:s3_smoothed_dual_func} \textit{in parallel} to obtain an approximation of $x^{*}(y^{0,j},t_0)$.
\item Evaluate $\nabla{d}_{\bar{\delta}}(y^{0,j},t_0)$ and $\nabla^2d_{\bar{\delta}}(y^{0,j},t_0)$ by \eqref{eq:s4_derivs_of_inex_sm_dual_func}.
\item Perform an inexact-perturbed damped Newton iteration: $y^{0,j+1} := y^{0,j} -
\lambda_j(1+\lambda_j)^{-1}\nabla^2d_{\bar{\delta}}(y^{0,j},t_0)^{-1}\nabla{d}_{\bar{\delta}}(y^{0,j},t_0)$.
\end{enumerate}
\item[]\texttt{End For}
\end{itemize}
\item[]\textbf{Phase 2.} \textit{Path-following iterations}
\begin{itemize}
\item[] Compute $\sigma \in (0,1)$. Set $k := 0$.
\item[] \texttt{For $k=0, 1, \dots$ perform:}
\begin{enumerate}
\item If $t_k \leq \varepsilon_d/\omega_{*}(\beta)$ then terminate.
\item Update $t_{k+1} := (1-\sigma)t_k$.
\item Solve \eqref{eq:s3_smoothed_dual_func} \textit{in parallel} to obtain an approximation of $x^{*}(y^k,t_{k+1})$.
\item Evaluate the quantities $\nabla{d}_{\bar{\delta}}(y^k,t_{k+1})$ and $\nabla^2d_{\bar{\delta}}(y^k, t_{k+1})$.
\item Perform an inexact-perturbed full-step Newton iteration $y^{k+1} := y^k - \nabla^2d_{\bar{\delta}}(y^k, t_{k+1})^{-1}\nabla d_{\bar{\delta}}(y^k,
t_{k+1})$.
\end{enumerate}
\item[] \texttt{End For}
\end{itemize}
\item[]\textbf{Output.} A $2\varepsilon_d$-approximate solution $y^{k}$ of \eqref{eq:s2_original_dual_prob}.
\end{itemize}
This algorithm is still conceptual. In the following subsections, we shall specify each step of this algorithmic framework in detail.

Let us emphasize an important point. In order to compute $d(y,t)$ we have to solve exactly the maximization problem in \eqref{eq:s3_smoothed_dual_func} or
equivalently, to solve the system of nonlinear equations \eqref{eq:s3_optimality_for_smoothed_dual_prob}. This requirement is impractical. In practice, we can
only solve this problem up to a desired accuracy $\bar{\delta} > 0$. Therefore, the theory of the path-following algorithm presented in
\cite{Kojima1993,Mehrotra2009,Necoara2009,Shida2008,Zhao2005} for solving \eqref{eq:s3_smoothed_dual_prob} may no longer be satisfied.
Here, we propose an inexact perturbed path-following algorithm for solving \eqref{eq:s3_smoothed_dual_prob}.
This algorithm allows us to solve inexactly the primal subproblem \eqref{eq:s3_smoothed_dual_func}. Consequently, inexact-perturbed Newton-type iterations are
performed, which means that not only inexact gradient but also inexact Hessian of $d(\cdot,t)$ are used.

\subsection{Computing inexact solution $\bar{x}_{\bar{\delta}}$}
Note that condition \eqref{eq:s4_approx_sol} can not be used in practice to compute $\bar{x}_{\bar{\delta}}$ since $x^{*}(y,t)$ is unknown.
We show how to compute $\bar{x}_{\bar{\delta}}$ such that \eqref{eq:s4_approx_sol} holds based on the optimality condition
\eqref{eq:s3_optimality_for_smoothed_dual_prob}.

For sake of notational simplicity, we abbreviate by $\bar{x}_{\bar{\delta}} := \bar{x}_{\bar{\delta}}(y, t)$ and $x^{*} := x^{*}(y,t)$.
The error of the approximate solution $\bar{x}_{\bar{\delta}}$ to $x^{*}$ is defined as
\begin{eqnarray}\label{eq:s4_sol_distance}
\delta(\bar{x}_{\bar{\delta}}, x^{*}) := \norm{\bar{x}_{\bar{\delta}}(y,t) - x^{*}(y,t)}_{x^{*}(y,t)}.
\end{eqnarray}
It follows from the definitions of $d(\cdot,t)$ and $d_{\bar{\delta}}(\cdot,t)$, and \eqref{eq:s3_optimality_for_smoothed_dual_prob} that
\begin{align*}
d(y,t) - d_{\bar{\delta}}(y,t) &= [c + A^Ty](x^{*} - \bar{x}_{\bar{\delta}}) - t[F(x^{*}) - F(\bar{x}_{\bar{\delta}})] \nonumber\\
& = -t[F(x^{*}) + \nabla F(x^{*})^T(\bar{x}_{\bar{\delta}} - x^{*}) - F(\bar{x}_{\bar{\delta}})].
\end{align*}
Since $F$ is self-concordant, by applying \cite[Theorems 4.1.7 and 4.1.8]{Nesterov2004}, and the definition of $\delta(\bar{x}_{\bar{\delta}},x^{*})$,
the above equality implies
\begin{equation}\label{eq:s4_bound_on_sm_dual_and_inex_dual_funs}
0 \leq t\omega(\delta(\bar{x}_{\bar{\delta}},x^{*})) \leq d(y,t) - d_{\bar{\delta}}(y,t) \leq t\omega_{*}(\delta(\bar{x}_{\bar{\delta}},x^{*})).
\end{equation}
Here, the last inequality holds if  $\delta(\bar{x}_{\bar{\delta}},x^{*}) < 1$.

Next, using again the optimality condition \eqref{eq:s3_optimality_for_smoothed_dual_prob} we have
\begin{align*}
E^c_{\bar{\delta}} := \norm{c + A^Ty - t\nabla F(\bar{x}_{\bar{\delta}})}_{x^c}^{*} &\overset{\tiny\eqref{eq:s3_optimality_for_smoothed_dual_prob}}{=}
t\norm{\nabla F(\bar{x}_{\bar{\delta}}) - \nabla
F(x^{*})}_{x^c}^{*} \nonumber\\
&\geq \frac{t}{\nu+2\sqrt{\nu}}\norm{\nabla F(\bar{x}_{\bar{\delta}}) - \nabla F(x^{*})}_{x^{*}}^{*},
\end{align*}
where the last inequality follows from \cite[Corollary 4.2.1]{Nesterov2004}.
Combining this inequality and \cite[Theorem 4.1.7]{Nesterov2004}, we obtain
\begin{align*}
\frac{\delta(\bar{x}_{\bar{\delta}},x^{*})^2}{1+\delta(\bar{x}_{\bar{\delta}},x^{*})} &\leq  \left[\nabla F(\bar{x}_{\bar{\delta}}) - \nabla
F(x^{*})\right]^T(\bar{x}_{\bar{\delta}}-x^{*}) \\
& \leq \norm{\nabla F(\bar{x}_{\bar{\delta}}) - \nabla F(x^{*})}_{x^{*}}^{*}\norm{\bar{x}_{\bar{\delta}} - x^{*}}_{x^{*}}\\
&\leq \frac{(\nu+2\sqrt{\nu})E^c_{\bar{\delta}}}{t}\delta(\bar{x}_{\bar{\delta}},x^{*}).
\end{align*}
Hence, we get
\begin{equation}\label{eq:s3_condition_of_x_delta}
\delta(\bar{x}_{\bar{\delta}},x^{*}) \leq \frac{(\nu+2\sqrt{\nu})E^c_{\bar{\delta}}}{t-(\nu+2\sqrt{\nu})E^c_{\bar{\delta}}},
\end{equation}
provided that $t > (\nu+2\sqrt{\nu})E^c_{\bar{\delta}}$.
Let us define an accuracy $\varepsilon_p$ for the primal subproblem \eqref{eq:s3_smoothed_dual_func} as
\begin{equation}\label{eq:s4_choice_tol_for_inex_case}
\varepsilon_p := \frac{\bar{\delta}t}{(\nu+2\sqrt{\nu})(1+\bar{\delta})} \geq 0.
\end{equation}
Then it follows from \eqref{eq:s3_condition_of_x_delta} that if
\begin{equation}\label{eq:s4_inex_primal_optimality}
E^c_{\bar{\delta}} = \norm{c + A^Ty - t\nabla F(\bar{x}_{\bar{\delta}})}_{x^c}^{*} \leq \frac{\bar{\delta}t}{(\nu+2\sqrt{\nu})(1+\bar{\delta})}
\end{equation}
then $\bar{x}_{\bar{\delta}}(y,t)$ satisfies \eqref{eq:s4_approx_sol}.

It remains to consider the distance from $d_{\delta}$ to $d^{*}_0$ when $t$ is sufficiently small.
Suppose that $t \leq \omega_{*}(\beta)^{-1}\varepsilon_d$. Then, by combining \eqref{eq:s3_eps_dual_sol} and
\eqref{eq:s4_bound_on_sm_dual_and_inex_dual_funs} we have
\begin{equation}\label{eq:s4_eps_approx_sol}
\abs{d_{\bar{\delta}}(y,t) - \phi^{*}} = \abs{d_{\bar{\delta}}(y,t) - d_0^{*}} \leq 2\left[1 +
\omega_{*}(\beta)^{-1}\omega_{*}(\bar{\delta})\right]\varepsilon_d,
\end{equation}
provided that $\bar{\delta} < 1$.

\begin{remark}\label{re:s4_choice_eps}
Since $E_{\bar{\delta}} := \norm{c + A^Ty - t\nabla F(\bar{x}_{\bar{\delta}})}_{\bar{x}_{\bar{\delta}}}^{*} \geq
(1-\bar{\delta})\norm{c + A^Ty - t\nabla F(\bar{x}_{\bar{\delta}})}^{*}_{x^{*}}$. By the same argument as before, we can show that if $E_{\bar{\delta}} \leq
\hat{\varepsilon}_p$, where $\hat{\varepsilon}_p := \frac{\bar{\delta}(1-\bar{\delta})t}{1+\bar{\delta}}$ then \eqref{eq:s4_approx_sol} holds. This rule can be
used to terminate the algorithms presented in the next sections.
\end{remark}

\subsection{Phase 2 - The path-following scheme with inexact-perturbed full-step Newton iterations}\label{subsec:inexact_NT_iteration}
Now, we analyze Steps 2-5 in Phase 2 of the algorithmic framework. In the path-following fashion, we only perform one inexact-perturbed full-step Newton (IPFNT)
iteration for each value of parameter $t$. In other words, the IPFNT iteration and the update of $t$ are simultaneously carried out. The parameter $t$ is
decreased by $t_{+} := t - \Delta t$, where $\Delta t > 0$. Hence, one step of the path-following method is performed as follows:
\begin{equation}\label{eq:s4_path_following_iter}
\begin{cases}
t_{+} := t - \Delta t,\\
y_{+} := y - \nabla^2d_{\bar{\delta}}(y,t_{+})^{-1}\nabla{d}_{\bar{\delta}}(y, t_{+}).
\end{cases}
\end{equation}
Since Newton method is invariant under linear transformations, by \eqref{eq:s4_inex_sm_dual_func}, the second line of \eqref{eq:s4_path_following_iter} is
equivalent to
\begin{equation}\label{eq:s4_inexact_fs_NT_invariant}
y_{+} := y - \nabla^2\tilde{d}_{\bar{\delta}}(y,t_{+})^{-1}\nabla\tilde{d}_{\bar{\delta}}(y,t_{+}).
\end{equation}

For sake of notational simplicity, we denote all the functions at $(y_{+}, t_{+})$ and $(y, t_{+})$ by the sub-index ``$_+$'' and ``$_1$'', respectively, and
at $(y,t)$ without index in the following analysis. More precisely, we denote by 
\begin{eqnarray*}
&&\begin{array}{lll}
&\bar {\lambda}_{+} :=  \bar{\lambda}_{\tilde{d}_{\bar{\delta}}(\cdot,t_{+})}(y_{+}), ~&~
\delta_{+} := \delta(\bar{x}_{\bar{\delta}+}, x^{*}_{+}) = \norm{\bar{x}_{\bar{\delta}}(y_{+},t_{+}) - x^{*}(y_{+},t_{+})}_{x^{*}(y_{+},t_{+})},\\
&\bar {\lambda}_1 :=  \bar{\lambda}_{\tilde{d}_{\bar{\delta}}(\cdot,t_{+})}(y), ~&~
\delta_1 := \delta(\bar{x}_{\bar{\delta}1}, x^{*}_1) = \norm{\bar{x}_{\bar{\delta}}(y,t_{+}) - x^{*}(y,t_{+})}_{x^{*}(y,t_{+})},\\
&\bar {\lambda} :=  \bar{\lambda}_{\tilde{d}_{\bar{\delta}}(\cdot,t)}(y), ~&~
\delta := \delta(\bar{x}_{\bar{\delta}}, x^{*}) = \norm{\bar{x}_{\bar{\delta}}(y,t) - x^{*}(y,t)}_{x^{*}(y,t)},
\end{array}\\
\textrm{and}~&&\mathrm{by}\\
&&\begin{array}{lll}
\Delta := \norm{\bar{x}_{\bar{\delta}}(y,t_{+}) - \bar{x}_{\bar{\delta}}(y,t)}_{\bar{x}_{\bar{\delta}}(y,t)} ~\mathrm{and}~
\Delta^{*} := \norm{x^{*}(y, t_{+}) - x^{*}(y,t)}_{x^{*}(y,t)}.
\end{array}
\end{eqnarray*}
Note that the above notation  does not cause any confusion since it can be recognized from the context.

\subsubsection{The main estimate}
Using the above notation, we provide a main estimate which will be used to analyze the convergence of the algorithm presented in Subsection
\ref{subsubsec:algorithm}. The proof of this result is postponed to Subsection \ref{subsec:proofs_of_technical_lemmas}.

\begin{lemma}\label{le:s4_main_inequalities}
Let $y\in Y$ be given and  $t> 0$. Let  $(y_{+},t_{+})$ be a pair generated by \eqref{eq:s4_path_following_iter}.
Suppose that $\delta_1 + 2\Delta + \bar{\lambda} < 1$, $ \delta_{+} < 1$ and $\xi := \frac{\Delta + \bar{\lambda}}{1-\delta_1 - 2\Delta - \bar{\lambda}}$. Then
\begin{eqnarray}\label{eq:s4_main_estimate}
\bar{\lambda}_{+} \leq \frac{1}{(1 - \delta_{+})}\left\{ \delta_{+} + \delta_1 + \xi^2 + \delta_1\left[(1-\delta_1)^{-2} + 2(1-\delta_1)^{-1}\right]\xi\right\}.
\end{eqnarray}
Moreover, the right-hand side of \eqref{eq:s4_main_estimate} is nondecreasing with respect to all variables $\delta_{+}$, $\delta_1$, $\Delta$ and
$\bar{\lambda}$.

In particular, if we set $\delta_{+} = 0$ and $\delta_1 = 0$, i.e. \eqref{eq:s3_smoothed_dual_func} is solved exactly, then $\bar{\lambda}_{+} = \lambda_{+}$,
$\bar{\lambda} = \lambda$ and \eqref{eq:s4_main_estimate} collapses to
\begin{equation}\label{eq:s4_main_estimate_for_exact_case}
\lambda_{+} \leq \left(\frac{\lambda + \Delta^{*}}{1-2\Delta^{*} - \lambda}\right)^2,
\end{equation}
provided that $\lambda + 2\Delta^{*} < 1$.
\end{lemma}

\subsubsection{Finding the maximum centering parameter $\beta^{*}$}
The key point of the path-following algorithm  is to determine the maximum value of $\beta \in (0, \beta^{*}) \subseteq (0,1)$ and appropriate values of
$\bar{\delta}$ and $\Delta$ such that if $\bar{\lambda} \leq \beta$, then $\bar{\lambda}_{+} \leq \beta$. We analyze the estimate \eqref{eq:s4_main_estimate} to
find these parameters.

First, let  $\beta \in (0, 1)$ such that $\bar{\lambda} \leq \beta$. Since the right-hand side of \eqref{eq:s4_main_estimate} is nondecreasing with respect to
all variables, if we define
\begin{equation*}
\varphi_{\bar{\delta}}(\bar{\xi}) := \frac{1}{1-\bar{\delta}}\left\{2\bar{\delta} + \bar{\xi}^2 + \bar{\delta}[(1-\bar{\delta})^{-2} +
2(1-\bar{\delta})^{-1}]\bar{\xi}\right\},
\end{equation*}
and $\bar{\xi} := \frac{\Delta + \beta}{1-\bar{\delta} - \beta - 2\Delta}$, then $\bar{\lambda}_{+} \leq \beta$ if $\varphi_{\bar{\delta}}(\bar{\xi}) \leq
\beta$.
This condition leads to $0 \leq \bar{\xi} \leq \frac{\sqrt{p^2 + 4q} - p}{2}$ and $0\leq \bar{\delta} \leq \frac{\beta}{\beta+2}$, where $p :=
\bar{\delta}[(1-\bar{\delta})^{-2} + 2(1-\bar{\delta})^{-1}]$ and $q := (1-\bar{\delta})\beta - 2\bar{\delta}$.

Now, let $\theta := \frac{\sqrt{p^2 + 4q} - p}{2} > 0$. Since $\bar{\xi} = \frac{\beta+\Delta}{1-\bar{\delta} - \beta - 2\Delta} \leq \theta$, we have
$(1+2\theta)\Delta \leq \theta(1-\bar{\delta}-\beta) - \beta$.
Thus, in order to ensure $\Delta > 0$, we require that $\theta = \frac{\sqrt{p^2 + 4q} - p}{2} > \frac{\beta}{1-\bar{\delta} - \beta}$. This condition leads to
\begin{equation}\label{eq:s4_cubic_polynomial}
\mathcal{P}(\beta) := c_0 + c_1\beta + c_2\beta^2 + c_3\beta^3 > 0,
\end{equation}
where $c_0 := -2\bar{\delta}(1-\bar{\delta})^2 \leq 0$, $c_1 := (1-\bar{\delta})[(1+\bar{\delta})^2-p] \geq 0$,
$c_2 := p - 3 - 2\bar{\delta}^2 + 2\bar{\delta} \leq 0$ and $c_3 := 1-\bar{\delta} > 0$.
By well-known characteristics of the cubic polynomial, we know that $\mathcal{P}(\beta)$ has three real roots if $18c_0c_1c_2c_3 - 4c_2^3c_0 + c_2^2c_1^2
- 4c_3c_1^3 - 27c_3^2c_0^2 \geq 0$. By numerical solution, the last condition leads to $0 \leq \bar{\delta} \leq \bar{\delta}_{\max}$, where
$\bar{\delta}_{\max} \approx 0.0432863855$.

Finally, we summarize the above analysis into the following theorem.

\begin{theorem}\label{th:s4_select_conditions}
Let $\bar{\delta}_{\max} = 0.0432863855$ and $0\leq \bar{\delta} \leq \bar{\delta}_{\max}$. Then $\mathcal{P}$ defined by
\eqref{eq:s4_cubic_polynomial} has
three nonnegative real roots $0 \leq \beta_{*} < \beta^{*} < \beta_3$.
Suppose that $\beta \in (\beta_{*}, \beta^{*})$ and $\bar{\Delta} := \frac{\theta(1-\bar{\delta}-\beta)-\beta}{1+2\theta} > 0$ where $\theta := \frac{\sqrt{p^2
+ 4q}-p}{2}$, and $p$ and $q$ are defined as above.
Then, for $0\leq \delta_{+} \leq \bar{\delta}$, $0\leq \delta_1 \leq \bar{\delta}$ and $0 \leq \Delta \leq \bar{\Delta}$, if $\bar{\lambda} \leq \beta$ then
$\bar{\lambda}_{+} \leq \beta$.
\end{theorem}

\begin{proof}
Note that the cubic polynomial $\mathcal{P}(\beta)$ has three real roots if $18c_0c_1c_2c_3 - 4c_2^3c_0 + c_2^2c_1^2 - 4c_3c_1^3 - 27c_3^2c_0^2 \geq 0$.
Numerically, this condition leads to $0 \leq \bar{\delta} \leq \bar{\delta}_{\max} = 0.0432863855$. Moreover, one can show that three roots $\beta_{*} <
\beta^{*} < 1 < \beta_3$  of $\mathcal{P}$ are nonnegative and $\mathcal{P}(\beta) > 0$ if $\beta \in (\beta_{*}, \beta^{*})$. However, $\mathcal{P}(\beta) > 0$
implies $\theta(1-\bar{\delta}-\beta)-\beta > 0$, where $\theta := \frac{\sqrt{p^2 + 4q}-p}{2}$. Thus, from the definition of $\bar{\xi}$, we have $0 \leq
\Delta \leq \bar{\Delta} := \frac{\theta(1-\bar{\delta}-\beta)-\beta}{1+2\theta} > 0$.
\end{proof}

In order to see the values of $\beta_{*}$, $\beta^{*}$ and $\bar{\Delta}$ varying with respect to the accuracy $\bar{\delta}$, we illustrate them in Figure
\ref{fig:quad_region}, where the left-hand side shows the values of $\beta_{*}$ (solid) and $\beta^{*}$ (dash) and the right-hand side shows the value of
$\bar{\Delta}$ varying with respect to $\bar{\delta}$ when $\beta$ is chosen by $\beta := \frac{\beta_{*}+\beta^{*}}{2}$ (dash) and $\beta :=
\frac{\beta^{*}}{4}$ (solid), respectively.
\begin{figure}[ht]
\centerline{\includegraphics[angle=0,height=3.8cm,width=13.0cm]{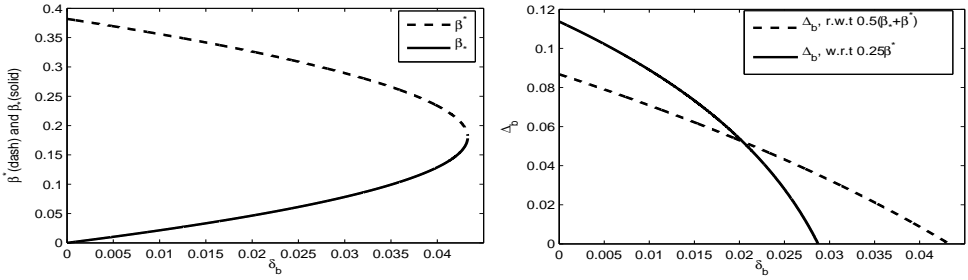}}
\caption{The values of $\beta_{*}$, $\beta^{*}$ and $\bar{\Delta}$ varying w.r.t $\bar{\delta}$.}
\label{fig:quad_region}
\end{figure}

\subsubsection{The update rule for the barrier parameter $t$}
It remains to quantify the decrement $\Delta{t}$ of the barrier parameter $t$.
From \eqref{eq:s3_optimality_for_smoothed_dual_prob} we have
\begin{eqnarray*}
c + A^Ty - t\nabla F(x^{*}) = 0 ~~\mathrm{and} ~~ c + A^Ty - t_{+}\nabla F(x^{*}_1) = 0,
\end{eqnarray*}
where $x^{*} := x^{*}(y,t)$ and $x^{*}_1 := x^{*}(y, t_{+})$ are defined as before.
Subtracting these equalities and then using $t_{+} = t - \Delta t$, we have $t_{+}[\nabla F(x^{*}_1) - \nabla F(x^{*})] = \Delta t\nabla F(x^{*})$.
Using this relation together with \cite[Theorem 4.1.7]{Nesterov2004} and $\norm{\nabla F(x^{*})}^{*}_{x^{*}}\leq \sqrt{\nu}$ (see \cite[inequality
4.2.4]{Nesterov2004}), we have
\begin{align*}
\frac{t_{+}\norm{x^{*}_1 - x^{*}}_{x^{*}}^2}{1 + \norm{x^{*}_1 - x^{*}}_{x^{*}}} &\leq t_{+}[\nabla F(x^{*}_1) - \nabla F(x^{*})]^T(x^{*}_1 - x^{*}) = \Delta
t\nabla F(x^{*})^T(x^{*}_1 - x^{*}) \nonumber\\
& \leq \Delta t\norm{\nabla F(x^{*})}^{*}_{x^{*}}\norm{x^{*}_1 - x^{*}}_{x^{*}} \leq \Delta t\sqrt{\nu}\norm{x^{*}_1 - x^{*}}_{x^{*}}. \nonumber
\end{align*}
By the definition of $\Delta^{*}$, if $t > (\sqrt{\nu} +
1)\Delta{t}$, then the above inequality leads to
\begin{equation}\label{eq:s4_choice_of_Delta_star}
\Delta^{*} \leq \bar{\Delta}^{*} := \frac{\sqrt{\nu}\Delta t}{t - (\sqrt{\nu} + 1)\Delta t}.
\end{equation}
Note that \eqref{eq:s4_choice_of_Delta_star} implies
\begin{equation}\label{eq:s4_bar_par_decrement}
\Delta{t} = \frac{\bar{\Delta}^{*}t}{\sqrt{\nu} + (\sqrt{\nu} + 1)\bar{\Delta}^{*}}.
\end{equation}
On the other hand, using the definitions of $\Delta$ and $\delta$, we have
\begin{eqnarray}\label{eq:s4_bar_par_decrement_estimate}
\Delta := \norm{\bar{x}_{\bar{\delta}1} - \bar{x}_{\bar{\delta}}}_{\bar{x}_{\bar{\delta}}}
&&\overset{\tiny\eqref{eq:s4_F_hessian_est}}{\leq}
\frac{1}{(1-\delta)}\Big[\norm{\bar{x}_{\bar{\delta}1} -
x^{*}_{1}}_{x^{*}} + \norm{x^{*}_{1} - x^{*}}_{x^{*}} + \norm{x^{*} - \bar{x}_{\bar{\delta}}}_{x^{*}}\Big] \nonumber\\
&& \leq \frac{1}{(1-\delta)}\left[\frac{\delta_{1}}{1-\Delta^{*}} + \Delta^{*} + \delta \right] \nonumber\\
[-1.5ex]\\[-1.5ex]
&& \overset{\tiny\eqref{eq:s4_choice_of_Delta_star}}{\leq} \frac{1}{(1-\delta)}\left[\frac{\delta_{1}}{1-\bar{\Delta}^{*}} + \bar{\Delta}^{*} + \delta\right]
\nonumber\\
&& \overset{\tiny{\delta, \delta_{1}\leq \bar{\delta}}}{\leq} \frac{1}{(1-\bar{\delta})}\left[\frac{\bar{\delta}}{1-\bar{\Delta}^{*}} + \bar{\Delta}^{*} +
\bar{\delta} \right].
\nonumber
\end{eqnarray}
Now, we need to find a condition such that $\Delta \leq \bar{\Delta}$, where $\bar{\Delta}$ is given in Theorem \ref{th:s4_select_conditions}. This condition
holds
if
$\frac{\bar{\delta}}{1-\Delta^{*}} + \Delta^{*} \leq (1-\bar{\delta})\bar{\Delta} - \bar{\delta}$ due to \eqref{eq:s4_bar_par_decrement_estimate}.
Since $\Delta^{*} \leq \bar{\Delta}^{*}$ due to \eqref{eq:s4_choice_of_Delta_star}, we impose a more relaxed condition
\begin{equation}\label{eq:s4_bar_par_decrement_factor}
0\leq \bar{\Delta}^{*} \leq \frac{1}{2}\left[(1-\bar{\delta})\bar{\Delta} - \bar{\delta} + 1 - \sqrt{((1-\bar{\delta})\bar{\Delta} - \bar{\delta} - 1)^2 +
4\bar{\delta}}\right],
\end{equation}
provided $\bar{\delta} \leq \frac{\bar{\Delta}}{1+\bar{\Delta}}$.
Thus, we can fix $\bar{\Delta}^{*}$ at
\begin{equation}\label{eq:s4_choice_of_Delta_bar_star}
\bar{\Delta}^{*} = \frac{1}{2}\left[(1-\bar{\delta})\bar{\Delta} - \bar{\delta} + 1 - \sqrt{((1-\bar{\delta})\bar{\Delta} - \bar{\delta} - 1)^2 +
4\bar{\delta}}\right].
\end{equation}
The update rule for the barrier parameter $t$ becomes
\begin{equation*}
t_{+} := (1 - \sigma)t = \left(1 - \frac{\bar{\Delta}^{*}}{\sqrt{\nu} + (\sqrt{\nu}+1)\bar{\Delta}^{*}}\right) = \frac{\sqrt{\nu}(\bar{\Delta}^{*} +
1)t}{\sqrt{\nu}(\bar{\Delta}^{*} + 1) +
\bar{\Delta}^{*}},
\end{equation*}
where $\sigma := \frac{\bar{\Delta}^{*}}{\sqrt{\nu} + \bar{\Delta}^{*}(\sqrt{\nu}+1)} \in (0, 1)$.

Finally, we show that the conditions given in Theorem \ref{th:s4_select_conditions}, \eqref{eq:s4_bar_par_decrement_factor} and
\eqref{eq:s4_choice_of_Delta_bar_star} are well-defined.
Indeed, let us fix $\bar{\delta} := 0.01$. Then we can compute the values of $\beta_{*}$ and $\beta^{*}$ as
\begin{equation*}
\beta_{*} \approx 0.021371 <  \beta^{*} \approx 0.356037.
\end{equation*}
Therefore, if we choose $\beta := \frac{\beta^{*}}{4} \approx 0.089009 > \beta_{*}$ then
\begin{equation*}
\bar{\Delta} \approx 0.089012,~\mathrm{and}~ \bar{\Delta}^{*} \approx 0.067399.
\end{equation*}

\subsubsection{The algorithm and its convergence}\label{subsubsec:algorithm}
Now, we are at the point to present the algorithm and its convergence. Before presenting the algorithm, we need to find a stopping criterion for the algorithm.
By using Lemma \ref{le:s4_basic_estimates}c), we have
\begin{eqnarray}\label{eq:s4_relations_of_lambda_bar_and_lambda}
\lambda \leq (1-\delta)^{-1}(\bar{\lambda} + \delta),
\end{eqnarray}
provided that $\delta  < 1$ and $\bar{\lambda} \leq \beta < 1$.
Consequently, if $\bar{\lambda} \leq (1-\bar{\delta})^{-1}(\beta+\bar{\delta})$ then $\lambda \leq \beta$. Let us define $\vartheta :=
(1-\bar{\delta})^{-1}(\beta+\bar{\delta})$. It follows from Lemma \ref{le:s3_dual_gap_estimate} that if $t\omega_{*}(\vartheta) \leq \varepsilon_d$ for a given
tolerance $\varepsilon_d > 0$, then $y$ is a $2\varepsilon_d$-solution of \eqref{eq:s2_original_dual_prob}.

The algorithmic framework presented in Subsection \ref{subsec:s4_alg_framework} is now described in detail as follows.

\noindent\rule[1pt]{\textwidth}{1.0pt}{~~}
\begin{algorithm}\label{alg:A1}{~}$($\textit{Path-following algorithm with $\mathrm{IPFNT}$ iterations}$)$
\end{algorithm}
\vskip -0.2cm
\noindent\rule[1pt]{\textwidth}{0.5pt}
\noindent\textbf{Initialization:} Perform the following steps:
\begin{enumerate}
\item Choose $\bar{\delta} \in [0, \bar{\delta}_{\max}]$ and compute $\beta_{*}$ and $\beta^{*}$ as the first and second roots of $\mathcal{P}$ defined by
\eqref{eq:s4_cubic_polynomial}, respectively.
\item Fix some $\beta \in (\beta_{*}, \beta^{*})$ (e.g. $\beta = \frac{1}{4}\beta^{*}$).
\item Choose an initial value $t = t_0 > 0$. 
\end{enumerate}
\noindent\textbf{Phase 1.}
Apply Algorithm \ref{alg:A1a} presented in Subsection \ref{subsec:Phase1_for_inexact_case} to find $y^0 \in Y$ such that
$\lambda_{\tilde{d}_{\bar{\delta}}(\cdot, t_0)}(y^0) \leq \beta$.
\vskip 0.1cm
\noindent\textbf{Phase 2.}\\
\noindent\textbf{Initialization:} Perform the following steps:
\begin{enumerate}
\item Given a tolerance $\varepsilon_d > 0$.
\item Compute $\bar{\Delta}$ as in Theorem \ref{th:s4_select_conditions}. Then, compute $\bar{\Delta}^{*}$ by \eqref{eq:s4_choice_of_Delta_bar_star}.
\item Compute the factor $\sigma := \frac{\bar{\Delta}^{*}}{\sqrt{\nu} + (\sqrt{\nu} + 1)\bar{\Delta}^{*}}$.
\item Compute the accuracy factor $\gamma := \frac{\bar{\delta}}{(\nu+2\sqrt{\nu})(1+\bar{\delta})}$.
\end{enumerate}
\noindent\textbf{Iteration:} Perform the following loop.\\
\texttt{For} $k=0,1,\cdots$ \texttt{do}
\begin{enumerate}
\item If $t_k \leq \frac{\varepsilon_d}{\omega_{*}(\vartheta)}$, where $\vartheta := (1-\bar{\delta})^{-1}(\beta + \bar{\delta})$,  then terminate.
\item Compute an accuracy for the primal subproblem $\varepsilon_k := \gamma t_k$.
\item Update $t_{k+1} := (1-\sigma)t_k$.
\item Solve approximately \eqref{eq:s2_original_dual_func} \textit{in parallel} up to the given tolerance $\varepsilon_k$ to obtain
$\bar{x}_{\bar{\delta}}(y^k, t_{k+1})$.
\item Compute $\nabla{d}_{\bar{\delta}}(y^k,t_{k+1})$ and $\nabla^2d_{\bar{\delta}}(y^k, t_{k+1})$ according to \eqref{eq:s4_derivs_of_inex_sm_dual_func}.
\item Update $y^{k+1}$ as $y^{k+1} := y^k - \nabla^2{d}_{\bar{\delta}}(y^k, t_{k+1})^{-1}\nabla{d}_{\bar{\delta}}(y^k, t_{k+1})$.
\end{enumerate}
\texttt{End of For}.
\vskip-0.2cm
\noindent\rule[1pt]{\textwidth}{1.0pt}

The core step of Phase 2 in Algorithm \ref{alg:A1} is Step 4, where we need to solve two convex optimization problems to compute the gradient vector and the
Hessian matrix of $d_{\bar{\delta}}(\cdot, t_{k+1})$ at Step 5. These quantities require an approximate solution $\bar{x}_{\bar{\delta}}(y^k, t_{k+1})$, the
gradient vector $\nabla F(\bar{x}_{\bar{\delta}}(y^k, t_{k+1}))$ and the Hessian matrix $\nabla^2 F(\bar{x}_{\bar{\delta}}(y^k, t_{k+1}))$, which can
also be computed \textit{in parallel}. Note that Step 4 actually requires to solve a system of nonlinear equations
\eqref{eq:s3_optimality_for_smoothed_dual_prob} (see Section \ref{sec:implementation_detail}  for more details). The update rule of $t$ at Step 3 can be done in
an adaptive way, where we can use $\norm{\nabla{F}(\bar{x}_{\bar{\delta}})}_{\bar{x}_{\bar{\delta}}}^{*}$ instead of its upper bound $\sqrt{\nu}$. For example,
we can use $\Delta{t} := \frac{\bar{\Delta}^{*}t}{R_{\bar{\delta}} + (R_{\bar{\delta}}+1)\bar{\Delta}^{*}}$ instead of \eqref{eq:s4_bar_par_decrement}, where
$R_{\bar{\delta}} := (1-\bar{\delta})^{-1}\left[\bar{\delta}(1-\bar{\delta})^{-1} +
\norm{\nabla{F}(\bar{x}_{\bar{\delta}})}^{*}_{\bar{x}_{\bar{\delta}}}\right]$.
The stopping criterion at Step 1 can be replaced by $\omega_{*}(\vartheta_k)t_k \leq \varepsilon_d$, where $\vartheta_k :=
(1-\bar{\delta})^{-1}[\lambda_{\tilde{d}_{\bar{\delta}}(\cdot,t_{k})}(y^{k}) + \bar{\delta}]$ due to Lemma \ref{le:s3_dual_gap_estimate}.

Let us define $\lambda_{k+1} := \lambda_{\tilde{d}_{\bar{\delta}}(\cdot,t_{k+1})}(y^{k+1})$ and $\lambda_k :=
\lambda_{\tilde{d}_{\bar{\delta}}(\cdot,t_k)}(y^k)$.
Then the local convergence of Algorithm \ref{alg:A1} is stated in the following theorem.

\begin{theorem}\label{th:s4_local_convergence}
Let $\{(y^k, t_k)\}$ be a sequence generated by Algorithm \ref{alg:A1}. Then the number of iterations $k_{\max}$ to obtain a $2\varepsilon_d$-solution of
\eqref{eq:s2_original_dual_prob} does not exceed
\begin{equation}\label{eq:s4_main_complexity_est}
k_{\max} := \left\lfloor \frac{\ln\left(\frac{\varepsilon_d}{t_0\omega_{*}(\vartheta)}\right)}{\ln(1-\sigma)} \right\rfloor + 1,
\end{equation}
where $\sigma = \frac{\bar{\Delta}^{*}}{\sqrt{\nu} + (\sqrt{\nu} + 1)\bar{\Delta}^{*}} \in (0, 1)$ and $\vartheta = (1-\bar{\delta})\beta - \bar{\delta} \in (0,
1)$.
\end{theorem}

\begin{proof}
Note that $y^k$ is a $2\varepsilon_d$-solution of \eqref{eq:s2_original_dual_prob} if $t_k \leq \frac{\varepsilon_d}{\omega_{*}(\vartheta)}$ due to Lemma
\ref{le:s3_dual_gap_estimate}, where
$\vartheta = (1-\bar{\delta})\beta - \bar{\delta}$. Since $t_k = (1-\sigma)^kt_0$ due to Step 3, we require $(1-\sigma)^k \leq
\frac{\varepsilon_d}{t_0\omega_{*}(\vartheta)}$. Consequently, we obtain \eqref{eq:s4_main_complexity_est}.
\end{proof}

\begin{remark}[\textbf{The worst-case complexity}]\label{eq:s4_compexity}
Since $(1-\sigma) = \left[1 + \frac{\bar{\Delta}^{*}}{\sqrt{\nu}(\bar{\Delta}^{*} + 1)}\right]^{-1}$ which implies $-\ln(1-\sigma) \sim
\frac{\bar{\Delta}^{*}}{\sqrt{\nu}(\bar{\Delta}^{*} + 1)}$.
It follows from Theorem \ref{th:s4_local_convergence}  that the complexity of Algorithm \ref{alg:A1} is $O(\sqrt{\nu}\ln\frac{t_0}{\varepsilon_d})$.
\end{remark}

\begin{remark}[\textbf{Linear convergence}]\label{re:s4_choice_of_new_bar_par_factor}
The rate of convergence of the sequence $\{t_k\}$  is linear and the contraction rate is not greater than $1-\sigma$. Note that if
$\lambda_{\tilde{d}_{\bar{\delta}}(\cdot, t)}(y) \leq \beta$, then it follows from \eqref{eq:s3_standard_sm_dual_func} that $\lambda_{d_{\bar{\delta}}(\cdot,
t})(y) \leq \beta\sqrt{t}$. Therefore, the sequence of Newton decrement $\{\lambda_{d(\cdot, t_k)}(y^k)\}_k$ of $d$ also converges linearly to zero with the
contraction factor less than or equal to $\sqrt{1 - \sigma}$.
\end{remark}

\begin{remark}[\textbf{Recovering the feasibility}]\label{re:s4_recovering_feasibility}
Since $\nabla d_{\bar{\delta}}(y, t) = A\bar{x}_{\bar{\delta}}(y,t) - b = t\nabla\tilde{d}_{\bar{\delta}}(y, t)$, we have $\dnorm{A\bar{x}_{\bar{\delta}}(y,t) -
b}_{y}^{*} = t\dnorm{\nabla\tilde{d}_{\bar{\delta}}(y, t)} = t\bar{\lambda} \leq t\beta$. If we define the inexact feasibility gap at
$\bar{x}_{\bar{\delta}}(y,t)$ as
\begin{equation*}
\mathcal{\bar{G}}_{\mathrm{feas}}(y,t) := \dnorm{A\bar{x}_{\bar{\delta}}(y,t) - b}_{y}^{*},
\end{equation*}
then $\mathcal{\bar{G}}_{\mathrm{feas}}(y,t) \leq t\beta$, which shows that $\mathcal{\bar{G}}_{\mathrm{feas}}(y,t)$ converges linearly to zero with the same
rate as $t$.
\end{remark}

\begin{remark}[\textbf{The inexactness in the IPFNT direction \eqref{eq:s4_path_following_iter}}]\label{re:s4_inexact_linear_algebra}
Note that we can apply an inexact method to solve the linear system \eqref{eq:s4_path_following_iter}. Under appropriate assumptions of the inexact term, we
can still prove the convergence of the algorithm. For more detail on inexact Newton methods, one can refer to the reference \cite{Wei2011}.
\end{remark}

\subsection{Phase 1 - Finding a starting point}\label{subsec:Phase1_for_inexact_case}
Phase 1 of the algorithmic framework aims to find $y^0\in Y$ such that $\lambda_{\tilde{d}_{\bar{\delta}}(\cdot, t)}(y^0) \leq \beta$. In this subsection, we
consider an inexact perturbed damped Newton (IPDNT) method for finding such a point $y^0$.

\subsubsection{Inexact perturbed damped Newton iteration}
Let us fix $t = t_0 > 0$ and choose an accuracy $\bar{\delta}\geq 0$. We assume that the current iterate $y\in Y$ is given, and we compute the next iterate
$y_{+}$ by applying the IPDNT iteration to  $d_{\bar{\delta}}(\cdot, t_0)$ as \begin{equation}\label{eq:s5_inexact_damped_NT_iter}
y_{+} := y - \alpha(y)\nabla^2d_{\bar{\delta}}(y,t_0)^{-1}\nabla d_{\bar{\delta}}(y,t_0),
\end{equation}
where $\alpha := \alpha(y) > 0$ is the step size which will  be chosen appropriately. Note that since \eqref{eq:s5_inexact_damped_NT_iter} is invariant under
linear transformation, it is equivalent to
\begin{equation}\label{eq:s5_inexact_damped_NT_invariant}
y_{+} := y - \alpha(y)\nabla^2\tilde{d}_{\bar{\delta}}(y,t_0)^{-1}\nabla\tilde{d}_{\bar{\delta}}(y,t_0),
\end{equation}
It follows from \eqref{eq:s3_standard_sm_dual_func} that $\tilde{d}(\cdot,t_0)$ is standard self-concordant, by \cite[Theorem 4.1.8]{Nesterov2004}, we have
\begin{equation}\label{eq:s5_main_inequality_1}
\tilde{d}(y_{+},t_0) \leq \tilde{d}(y, t_0) + \nabla\tilde{d}(y, t_0)^T(y_{+} - y) + \omega_{*}(\norm{y_{+}-y}_y),
\end{equation}
provided that $\norm{y_{+}-y}_y < 1$.
On the other hand, from \eqref{eq:s4_bound_on_sm_dual_and_inex_dual_funs}, it implies that
\begin{equation}\label{eq:s5_main_inequality_2}
0 \leq \omega(\delta(\bar{x}_{\bar{\delta}},x^{*})) \leq \tilde{d}(y,t_0) - \tilde{d}_{\bar{\delta}}(y,t_0) \leq
\omega_{*}(\delta(\bar{x}_{\bar{\delta}},x^{*})),
\end{equation}
which is an approximation between $\tilde{d}(\cdot,t_0)$ and $\tilde{d}_{\bar{\delta}}(\cdot, t_0)$.
In order to analyze the convergence of the IPDNT iteration \eqref{eq:s5_inexact_damped_NT_iter} we denote by
\begin{eqnarray}\label{eq:s5_short_notation}
&&\hat{\delta}_{+} := \norm{\bar{x}_{\bar{\delta}}(y_{+}, t_0) - x^{*}(y_{+},t_0)}_{x^{*}(y_{+},t_0)}, \nonumber\\
&&\hat{\delta} := \norm{\bar{x}_{\bar{\delta}}(y,t_0) - x^{*}(y,t_0)}_{x^{*}(y,t_0)}, \\
&&\bar{\lambda}_0 := \lambda_{\tilde{d}_{\bar{\delta}}(\cdot,t_0)}(y) = \alpha(y)\dnorm{y_{+} - y}_y, \nonumber
\end{eqnarray}
the solution differences of $d(\cdot,t_0)$ and $d_{\bar{\delta}}(\cdot, t_0)$ and the Newton decrement of $\tilde{d}_{\bar{\delta}}(\cdot, t_0)$, respectively.

\subsubsection{Finding the step size $\alpha(y)$}
Now, we find an appropriate step size $\alpha(y) \in (0, 1]$ such that the sequence generated by \eqref{eq:s5_inexact_damped_NT_invariant} converges to $y^0$.
Let $p := y_{+} - y$. From \eqref{eq:s5_main_inequality_1} and \eqref{eq:s5_main_inequality_2}, we have
\begin{align}\label{eq:s5_main_inequality_3}
\tilde{d}_{\bar{\delta}}(y_{+}, t_0) &\overset{\tiny{\eqref{eq:s5_main_inequality_1}}}{\leq} \tilde{d}(y_{+}, t_0)
\overset{\tiny{\eqref{eq:s5_main_inequality_2}}}{\leq} \tilde{d}(y, t_0) + \nabla\tilde{d}(y, t_0)^T(y_{+}-y) + \omega_{*}(\norm{y_{+}-y}_y) \nonumber\\
&\overset{\tiny{\eqref{eq:s5_main_inequality_1}}}{\leq} \tilde{d}_{\bar{\delta}}(y, t_0) + \nabla\tilde{d}(y, t_0)^T(y_{+}-y) + \omega_{*}(\norm{y_{+}-y}_y) +
\omega_{*}(\hat{\delta})
\nonumber\\
&=\tilde{d}_{\bar{\delta}}(y, t_0) \!+\! \nabla\tilde{d}_{\bar{\delta}}(y, t_0)^T\!p + \big[\nabla\tilde{d}(y, t_0) \!-\! \nabla\tilde{d}_{\bar{\delta}}(y,
t_0)\big]^T\!\!p  + \omega_{*}(\norm{p}_y) + \omega_{*}(\hat{\delta}) \\
&\overset{\tiny{\eqref{eq:s5_inexact_damped_NT_iter}}}{\leq} \tilde{d}_{\bar{\delta}}(y, t_0) -\alpha\bar{\lambda}_0^2 + \norm{\nabla\tilde{d}(y, t_0)
-\nabla\tilde{d}_{\bar{\delta}}(y, t_0)}_{y}^{*}\norm{p}_y  + \omega_{*}(\norm{p}_y) + \omega_{*}(\hat{\delta}) \nonumber\\
&\overset{\tiny{\eqref{eq:s4_basic_est_on_diff_grads_of_dual_fun}}}{\leq} \tilde{d}_{\bar{\delta}}(y, t_0) -\alpha\bar{\lambda}_0^2 + \hat{\delta}\norm{p}_y +
\omega_{*}(\norm{p}_y) + \omega_{*}(\hat{\delta}) \nonumber.
\end{align}
Furthermore, from \eqref{eq:s4_basic_est_on_Hessian_of_F} and the definition of $\nabla^2\tilde{d}$ and $\nabla^2\tilde{d}_{\bar{\delta}}$, we have
\begin{equation*}
(1-\hat{\delta})\nabla^2\tilde{d}_{\bar{\delta}}(y,t_0) \preceq \nabla^2\tilde{d}(y,t_0) \preceq (1-\hat{\delta})^{-2}\nabla^2\tilde{d}_{\bar{\delta}}(y,t_0).
\end{equation*}
This inequality implies that
\begin{equation*}
(1-\hat{\delta})\dnorm{p}_y \leq \norm{p}_y \leq (1-\hat{\delta})^{-1}\dnorm{p}_y.
\end{equation*}
Combining this inequality, \eqref{eq:s5_inexact_damped_NT_invariant} and the definition of $\bar{\lambda}_0$ in \eqref{eq:s5_short_notation} we get
\begin{equation*}
\alpha(1-\hat{\delta})\bar{\lambda}_0 \leq \norm{p}_y \leq \alpha(1-\hat{\delta})^{-1}\bar{\lambda}_0.
\end{equation*}
Let us assume that $\alpha\bar{\lambda}_0 + \hat{\delta} < 1$. By substituting the right-hand side of this inequality  into \eqref{eq:s5_main_inequality_3} and
observing  that the right hand side of \eqref{eq:s5_main_inequality_3} is nondecreasing with respect to $\norm{p}_y$, we get
\begin{align}\label{eq:s5_main_inequality_4}
\tilde{d}_{\bar{\delta}}(y_{+},t_0) \leq \tilde{d}_{\bar{\delta}}(y,t_0) -\alpha\bar{\lambda}_0^2 + \frac{\alpha\bar{\lambda}\hat{\delta}}{1-\hat{\delta}} +
\omega_{*}\left(\frac{\alpha\bar{\lambda}_0}{1-\hat{\delta}}\right) + \omega_{*}(\hat{\delta}).
\end{align}
Now, let us simplify the last terms of \eqref{eq:s5_main_inequality_4} which we denote by $T$ as follows.
\begin{align}\label{eq:s5_main_inequality_5}
T &:= -\alpha\bar{\lambda}_0^2 + \frac{\alpha\bar{\lambda}_0\hat{\delta}}{1-\hat{\delta}} + \omega_{*}\left(\frac{\alpha\bar{\lambda}_0}{1-\hat{\delta}}\right)
+ \omega_{*}(\hat{\delta}) \nonumber\\
& = -\alpha\bar{\lambda}_0^2 + \frac{\alpha\bar{\lambda}_0\hat{\delta}}{1-\hat{\delta}} - \frac{\alpha\bar{\lambda}_0}{1-\hat{\delta}} -
\ln\left(1-\frac{\alpha\bar{\lambda}_0}{1-\hat{\delta}}\right)
- \hat{\delta} - \ln(1-\hat{\delta}) \nonumber\\
& = -\alpha\bar{\lambda}_0^2 - (\alpha\bar{\lambda}_0 + \hat{\delta}) - \ln\left[1-(\alpha\bar{\lambda}_0 + \hat{\delta})\right] \\
& = -\alpha\bar{\lambda}_0^2 + \omega_{*}(\alpha\bar{\lambda}_0 + \hat{\delta}). \nonumber
\end{align}
Suppose that we can choose $\eta > 0$ such that $\alpha\bar{\lambda}_0^2 - \omega_{*}(\alpha\bar{\lambda}_0 + \hat{\delta}) = \omega(\eta)$.
This requirement leads to $\alpha\bar{\lambda}_0^2 = (\alpha\bar{\lambda}_0 + \hat{\delta})\left[\alpha(\bar{\lambda}_0 + \bar{\lambda}_0) +
\hat{\delta}\right]$ which is equivalent to
\begin{equation}\label{eq:s5_main_inequality_6}
\alpha = \frac{(1-\hat{\delta})\bar{\lambda}_0 - 2\hat{\delta} + \sqrt{(1-\hat{\delta})^2\bar{\lambda}_0^2 -
4\hat{\delta}\bar{\lambda}_0}}{2\bar{\lambda}_0(1+\bar{\lambda}_0)},
\end{equation}
provided that $0 \leq \hat{\delta} < \bar{\hat{\delta}} := \frac{2 + \bar{\lambda}_0 - 2\sqrt{1+\bar{\lambda}_0}}{\bar{\lambda}_0}$.
Consequently, we deduce
\begin{equation}\label{eq:s5_main_inequality_7}
\eta = \frac{\bar{\lambda}_0\left[(1-\hat{\delta})\bar{\lambda}_0 - 2\hat{\delta} + \sqrt{(1-\hat{\delta})^2\bar{\lambda}_0^2 -
4\hat{\delta}\bar{\lambda}_0}\right]}{(1 + \hat{\delta})\bar{\lambda}_0
+ \sqrt{(1 - \hat{\delta})^2\bar{\lambda}_0^2 - 4\hat{\delta}\bar{\lambda}_0}}.
\end{equation}
Note that if $\hat{\delta} = 0$, then $\alpha = \frac{1}{1+\bar{\lambda}_0}$ and $\eta = \bar{\lambda}_0$. The IPDNT iteration
\eqref{eq:s5_inexact_damped_NT_iter} becomes the exact damped Newton iteration as in \cite{Nesterov2004}.

We assume that $\bar{\lambda}_0 \geq \beta$ for a given $\beta \in (0, 1)$. Let us fix $\bar{\hat{\delta}}$ such that
\begin{equation}\label{eq:s5_choice_of_delta_bar}
0 < \bar{\hat{\delta}} < \hat{\delta}^{*} := \frac{2+\beta - 2\sqrt{1+\beta}}{\beta} = \frac{\beta}{2+\beta+2\sqrt{1+\beta}}.
\end{equation}
Next, we choose step size $\alpha$ as
\begin{equation}\label{eq:s5_choice_of_step_size}
\alpha(y) := \frac{(1-\bar{\hat\delta})\bar{\lambda}_0 - 2\bar{\hat\delta} + \sqrt{(1-\bar{\hat\delta})^2\bar{\lambda}_0^2 -
4\bar{\hat\delta}\bar{\lambda}_0}}{2\bar{\lambda}_0(1 + \bar{\lambda}_0)} \in (0, 1).
\end{equation}
Then the IPDNT iteration \eqref{eq:s5_inexact_damped_NT_iter} with $\alpha(y)$ given as \eqref{eq:s5_choice_of_step_size} generated a new point $y_{+}$ such
that
\begin{align}\label{eq:sec_descent_direction}
\tilde{d}_{\bar{\delta}}(y_{+}, t_0) \leq \tilde{d}_{\bar{\delta}}(y, t_0) - \omega(\underline{\eta}),
\end{align}
where
\begin{equation}\label{eq:s5_descent_decrement}
\underline{\eta} := \frac{\beta\left[(1-\bar{\hat\delta})\beta - 2\bar{\hat\delta} + \sqrt{(1-\bar{\hat\delta})^2\beta^2 -
4\bar{\hat\delta}\beta}\right]}{(1+\bar{\hat\delta})\beta + \sqrt{(1-\bar{\hat\delta})^2\beta^2 -
4\bar{\hat\delta}\beta}} \in (0, 1).
\end{equation}
Finally, let us estimate the constant $\underline{\eta}$ for the case $\beta \approx 0.089009$.
We first obtain $\hat{\delta}^{*} \approx 0.02131$. Let $\bar{\hat\delta} = \frac{1}{2}\hat{\delta}^{*} \approx 0.010657$. Then we get
$\underline{\eta} \approx 0.0754963$. Consequently, $\omega(\underline{\eta}) \approx 0.003002$.

\subsubsection{The algorithm and its worst-case complexity}
In summary, the algorithm for finding $y^0 \in Y$ is presented in detail as follows.

\noindent\rule[1pt]{\textwidth}{1.0pt}{~~}
\begin{algorithm}\label{alg:A1a}{~}$($\textit{Finding a starting point $y^0 \in Y$}$)$
\end{algorithm}
\vskip -0.2cm
\noindent\rule[1pt]{\textwidth}{0.5pt}
\noindent\textbf{Initialization:} Perform the following steps:
\begin{enumerate}
\item Input $\beta\in (\beta_{*}, \beta^{*})$ and $t_0 > 0$ as desired (e.g. $\beta = \frac{1}{4}\beta^{*} \approx 0.089009$).
\item Take an arbitrary point $y^{0,0}\in Y$.
\item Compute $\hat{\delta}^{*} := \frac{\beta}{2+\beta+2\sqrt{1+\beta}}$ and fixed $\bar{\hat\delta} \in (0,\hat{\delta}^{*})$ (e.g. $\bar{\hat\delta} =
0.5\hat{\delta}^{*}$).
\item Compute an accuracy $\varepsilon_p := \frac{t_0\bar{\hat\delta}}{2(\nu+2\sqrt{\nu})(1+\bar{\hat\delta})}$.
\end{enumerate}
\noindent\textbf{Iteration: } Perform the following loop.\\
\texttt{For} $j=0,1,\cdots$ \texttt{do}
\begin{enumerate}
\item Solve approximately the primal subproblem \eqref{eq:s3_smoothed_dual_func} \textit{in parallel} up to the accuracy
$\varepsilon_p$ to obtain  $\bar{x}_{\bar{\delta}}(y^{0,j},t_0)$.
\item Compute $\bar{\lambda}_j := \bar{\lambda}_{\tilde{d}_{\bar{\delta}}(\cdot, t_0)}(y^{0,j})$ .
\item If $\bar{\lambda}_j \leq \beta$ then set $y^0 := y^{0,j}$ and terminate.
\item Update $y^{0,j+1}$ as $y^{0,j+1} := y^{0,j} - \alpha_j\nabla^2{d}_{\bar{\delta}}(y^{0,j}, t_0)^{-1}\nabla{d}_{\bar{\delta}}(y^{0,j}, t_0)$,
where $\alpha_j \in (0, 1]$ is computed by
\begin{equation*}
\alpha_j := \frac{(1-\bar{\hat\delta})\bar{\lambda}_j - 2\bar{\hat\delta} + \sqrt{(1-\bar{\hat\delta})^2\bar{\lambda}^2_j -
4\bar{\hat\delta}\bar{\lambda}_j}}{2\bar{\lambda}_j(1 + \bar{\lambda}_j)}.
\end{equation*}
\end{enumerate}
\texttt{End of For}.\\
\vskip-0.3cm
\noindent\rule[1pt]{\textwidth}{1.0pt}
The convergence of this algorithm is stated in the following theorem.

\begin{theorem}\label{th:s5_Phase1_convergence}
The number of iterations to terminate Algorithm \ref{alg:A1a} does not exceed
\begin{equation}\label{eq:s5_Phase1_complexity}
J_{\max} := \left\lfloor\frac{d_{\bar{\delta}}(y^{0,0}, t_0) - d^{*}(t_0) + \omega_{*}(\bar{\hat\delta})}{t_0\omega(\underline{\eta})}\right\rfloor + 1,
\end{equation}
where $d^{*}(t_0) = \displaystyle\min_{y\in Y}d(y,t_0)$ and $\underline{\eta}$ is given by \eqref{eq:s5_descent_decrement}.
\end{theorem}

\begin{proof}
Summing up \eqref{eq:sec_descent_direction} from $j=0$ to $j = k$ and then using  \eqref{eq:s5_main_inequality_2} we have
\begin{align*}
0 &\leq \tilde{d}(y^{0,k},t_0) - \tilde{d}^{*}(t_0) \leq \tilde{d}_{\bar{\delta}}(y^{0,k},t_0) + \omega_{*}(\bar{\hat\delta}) - \tilde{d}^{*}(t_0) \nonumber\\
& \leq \tilde{d}_{\bar{\delta}}(y^{0,0},t_0) + \omega_{*}(\bar{\hat\delta}) - \tilde{d}^{*}(t_0) - k\omega(\underline{\eta}).
\end{align*}
This inequality together with \eqref{eq:s3_standard_sm_dual_func} and \eqref{eq:s4_standard_inx_sm_dual_func} imply
\begin{align*}
k \leq \frac{d_{\bar{\delta}}(y^{0,0}, t_0) - d^{*}(t_0) + \omega_{*}(\bar{\hat\delta})}{t_0\omega(\underline{\eta})}.
\end{align*}
Hence, the maximum number of iterations in Algorithm \ref{alg:A1a} does not exceed $J_{\max}$ defined by \eqref{eq:s5_Phase1_complexity}.
\end{proof}

Since $d^{*}(t_0)$ is not available, the number $J_{\max}$ in \eqref{eq:s5_Phase1_complexity} only gives an upper bound for Algorithm \ref{alg:A1a}. However, in
this algorithm, we do not use  $J_{\max}$ as a stopping criterion.

\subsection{The proof of Lemma \ref{le:s4_main_inequalities}}\label{subsec:proofs_of_technical_lemmas}
First, we prove the following lemma which will be used to prove the main inequality in Lemma \ref{le:s4_main_inequalities}.

\begin{lemma}\label{le:s4_basic_estimates}
Suppose that Assumptions \aref{as:A1} and \aref{as:A2} are satisfied. Then
\begin{itemize}
\item[$\mathrm{a)}$] $\nabla^2\tilde{d}$ and $\nabla^2\tilde{d}_{\bar{\delta}}$ defined by \eqref{eq:s3_deriv_of_sm_dual_func} and
\eqref{eq:s4_derivs_of_inex_sm_dual_func}, respectively, guarantee
\begin{eqnarray}\label{eq:s4_basic_est_on_Hessian_of_dual_fun}
(1-\delta_{+})^2\nabla^2\tilde{d}(y_{+},t_{+}) \preceq \nabla^2\tilde{d}_{\bar{\delta}}(y_{+},t_{+})  \preceq (1-\delta_{+})^{-2}\nabla^2\tilde{d}(y_{+},t_{+}),
\nonumber
\end{eqnarray}
where $\delta_{+} < 1$ defined by \eqref{eq:s4_sol_distance}.

\item[$\mathrm{b)}$] Moreover, one has
\begin{equation}\label{eq:s4_basic_est_on_diff_grads_of_dual_fun}
\norm{\nabla\tilde{d}_{\bar{\delta}}(y,t) - \nabla\tilde{d}(y,t)}^{*}_{y} \leq \norm{\bar{x}_{\bar{\delta}} - x^{*}}_{x^{*}}.
\end{equation}

\item[$\mathrm{c)}$] If $\Delta < 1$ then
\begin{equation}\label{eq:s4_basic_est_on_NT_decrements}
\bar{\lambda}_1 \leq \frac{\Delta + \bar{\lambda}}{1-\Delta}.
\end{equation}
\end{itemize}
\end{lemma}

\begin{proof}
Since $F$ is standard self-concordant, for any $z\in W^0(x, 1)$, it follows from \cite[Theorem 4.1.6]{Nesterov2004} that
\begin{equation}\label{eq:s4_F_hessian_est}
(1 - \norm{z-x}_x)^2\nabla^2F(x) \preceq \nabla^2F(z) \preceq \frac{1}{(1 - \norm{z-x}_x)^2}\nabla^2F(x).
\end{equation}
Since $\nabla^2F(x)$ is symmetric positive definite, by applying \cite[Proposition 8.6.6]{Bernstein2005} to two matrices $\frac{1}{(1 -
\norm{z-x}_x)^2}\nabla^2F(x)$ and $\nabla^2F(z)$, and then
to two matrices $(1 - \norm{z-x}_x)^2\nabla^2F(x)$ and $\nabla^2F(z)$ we obtain
\begin{eqnarray}\label{eq:s4_basic_est_on_Hessian_of_F}
(1 - \norm{z-x}_x)^2A\nabla^2F(x)^{-1}A^T &&\preceq A\nabla^2F(z)^{-1}A^T \nonumber\\
[-1.5ex]\\[-1.5ex]
&&\preceq (1-\norm{z-x}_x)^{-2}A\nabla^2F(x)^{-1}A^T. \nonumber
\end{eqnarray}
Using again \cite[Proposition 8.6.6]{Bernstein2005} for \eqref{eq:s4_basic_est_on_Hessian_of_F} we get
\begin{eqnarray}\label{eq:s4_basic_est_on_Hessian_of_F_2}
&&(1 - \norm{z-x}_x)^2A^T[A\nabla^2F(x)^{-1}A^T]^{-1}A \preceq A^T[A\nabla^2F(z)^{-1}A^T]^{-1}A \nonumber\\
[-1.5ex]\\[-1.5ex]
&& \preceq (1-\norm{z-x}_x)^{-2}A^T[A\nabla^2F(x)^{-1}A^T]^{-1}A. \nonumber
\end{eqnarray}
Now, using \eqref{eq:s3_deriv_of_sm_dual_func} and \eqref{eq:s3_standard_sm_dual_func}, we have $\nabla^2\tilde{d}(y,t) =
\frac{1}{t^2}A\nabla^2F(x^{*})^{-1}A^T$. Alternatively, using \eqref{eq:s4_derivs_of_inex_sm_dual_func} and \eqref{eq:s4_standard_inx_sm_dual_func}, we get
$\nabla^2\tilde{d}_{\bar{\delta}}(y,t) =
\frac{1}{t^2}A\nabla^2F(\bar{x}_{\bar{\delta}})^{-1}A^T$.
Substituting these relations with $x = x^{*}_{+}$ and $z = \bar{x}_{\bar{\delta}+}$ into
\eqref{eq:s4_basic_est_on_Hessian_of_F} and noting that $\delta_{+} = \delta(\bar{x}_{+},x^{*}_{+})$ defined by \eqref{eq:s4_sol_distance}, we
obtain \eqref{eq:s4_basic_est_on_Hessian_of_dual_fun}.

Next, we prove b). For any $x\in\dom(F)$, the Hessian matrix $\nabla^2F(x)$ is symmetric positive definite. Let us define
\begin{equation*}
M(x) := \begin{bmatrix}\nabla^2F(x) & A^T\\ A & A\nabla^2F(x)^{-1}A^T\end{bmatrix}.
\end{equation*}
First, we show that $M(x)$ is positive definite. Indeed, for any $z=(u,v)\in\mathbb{R}^n\times\mathbb{R}^m$,
we have
\begin{eqnarray*}
z^TM(x)z &&= u^T\nabla^2F(x)u + u^TA^Tv + v^TAu + v^TA\nabla^2F(x)^{-1}A^Tv \nonumber\\
&& = \norm{\nabla^2F(x)^{1/2}u}^2 + 2(\nabla^2F(x)^{1/2}u)^T(\nabla^2F(x)^{-1/2}A^Tv) + \norm{\nabla^2F(x)^{-1/2}A^Tv}^2 \nonumber\\
&& = \norm{\nabla^2F(x)^{1/2}u + \nabla^2F(x)^{-1/2}A^Tv}^2 \geq 0, \nonumber
\end{eqnarray*}
which shows that $M(x) \succeq 0$.
Now, since $A$ is full-row rank, $A\nabla^2F(x)^{-1}A^T$ is also symmetric positive definite. By applying Schur's complement to $M(x)$ \cite{Bernstein2005}, we
obtain
\begin{equation}\label{eq:s4_basic_est_on_Hessian_of_F_3}
A^T[A\nabla^2F(x)^{-1}A^T]^{-1}A \preceq \nabla^2F(x).
\end{equation}
To prove \eqref{eq:s4_basic_est_on_diff_grads_of_dual_fun} we note that $\nabla{d}_{\bar{\delta}}(y,t) - \nabla{d}(y,t) = A(\bar{x}_{\bar{\delta}} - x^{*})$.
Thus $\nabla\tilde{d}_{\bar{\delta}}(y,t)
-\nabla\tilde{d}(y,t) = \frac{1}{t}A(\bar{x}_{\bar{\delta}} - x^{*})$. This implies
\begin{align*}
\left[\norm{\nabla\tilde{d}_{\bar{\delta}}(y,t) - \nabla\tilde{d}(y,t)}_{y}^{*}\right]^2 &= \frac{1}{t^2}(\bar{x}_{\bar{\delta}} -
x^{*})^TA^T\nabla^2\tilde{d}(y,t)^{-1}A(\bar{x}_{\bar{\delta}} -
x^{*}) \\
& \overset{\tiny\eqref{eq:s3_deriv_of_sm_dual_func},\eqref{eq:s3_standard_sm_dual_func}}{=}
(\bar{x}_{\bar{\delta}} - x^{*})^TA^T[A\nabla^2F(x^{*})^{-1}A^T]^{-1}A(\bar{x}_{\bar{\delta}} - x^{*}) \\
&\overset{\tiny\eqref{eq:s4_basic_est_on_Hessian_of_F_3}}{\leq} (\bar{x}_{\bar{\delta}} - x^{*})^T\nabla^2F(x^{*})(\bar{x}_{\bar{\delta}} - x^{*}) \\
&= \norm{\bar{x}_{\bar{\delta}} - x^{*}}_{x^{*}}^2,
\end{align*}
which is indeed \eqref{eq:s4_basic_est_on_diff_grads_of_dual_fun}.

Finally, we prove \eqref{eq:s4_basic_est_on_NT_decrements}.
By using the definitions of $\nabla\tilde{d}_{\bar{\delta}}(\cdot,t_{+})$ and $\nabla^2\tilde{d}_{\bar{\delta}}(\cdot,t_{+})$ in
\eqref{eq:s4_derivs_of_inex_sm_dual_func}, of $\tilde{d}_{\bar{\delta}}(\cdot,t_{+})$ in
\eqref{eq:s4_standard_inx_sm_dual_func}, for any feasible point $\hat{x}$ of \eqref{eq:s2_sep_CP2}, it follows from the definition of $\bar{\lambda}_1$ in
\eqref{eq:s4_inex_dual_Newton_decrement} and $A\hat{x} = b$ that
\begin{eqnarray}\label{eq:s4_proof1_est1}
\bar{\lambda}_1^2 && = \left[\dnorm{\nabla\tilde{d}_{\bar{\delta}}(y,t_{+})}^{*}_y \right]^2 \nonumber\\
&& \overset{\tiny\eqref{eq:s4_inex_dual_Newton_decrement}}{=}
\nabla\tilde{d}_{\bar{\delta}}(y,t_{+})\nabla^2\tilde{d}_{\bar{\delta}}(y,t_{+})^{-1}\nabla\tilde{d}_{\bar{\delta}}(y,t_{+}) \nonumber\\
[-1.5ex]\\[-1.5ex]
&& \overset{\tiny\eqref{eq:s4_standard_inx_sm_dual_func}}{=}
\frac{1}{t_{+}}\nabla{d}_{\bar{\delta}}(y,t_{+})\nabla^2{d}_{\bar{\delta}}(y,t_{+})^{-1}\nabla{d}_{\bar{\delta}}(y,t_{+}) \nonumber\\
&& \overset{\tiny\eqref{eq:s4_derivs_of_inex_sm_dual_func}}{=} (\bar{x}_{\bar{\delta}1} -
\hat{x})^TA^T\left[A\nabla^2F(\bar{x}_{\bar{\delta}1})^{-1}A^T\right]^{-1}A(\bar{x}_{\bar{\delta}1} - \hat{x}).
\nonumber
\end{eqnarray}
Since $\Delta = \norm{\bar{x}_{\bar{\delta}1} - \bar{x}_{\bar{\delta}}}_{\bar{x}_{\bar{\delta}}} < 1$ by assumption, it implies that
$\bar{x}_{\bar{\delta}1} \in W^0(\bar{x}_{\bar{\delta}},1)$.
Applying the right-hand side of \eqref{eq:s4_basic_est_on_Hessian_of_F_2} with $x = \bar{x}_{\bar{\delta}}$ and $z  = \bar{x}_{\bar{\delta}1}$, it implies that
\begin{eqnarray}\label{eq:s4_proof1_est2}
\bar{\lambda}_1^2 \leq \frac{1}{(1 - \Delta)^2}(\bar{x}_{\bar{\delta}1} -
\hat{x})^TA^T\left[A\nabla^2F(\bar{x}_{\bar{\delta}})^{-1}{\!\!\!}A^T\right]^{-1}{\!\!\!\!}A(\bar{x}_{\bar{\delta}1} - \hat{x}).
\end{eqnarray}
Now, for any symmetric positive semidefinite matrix $Q$ in $\mathbb{R}^{n\times n}$ and $u, v\in\mathbb{R}^n$, one can easily show that
\begin{align}\label{eq:s4_proof1_est3}
(u + v)^TQ(u + v) \leq \left[\sqrt{u^TQu} + \sqrt{v^TQv}\right]^2. 
\end{align}
Since $H_{\bar{\delta}} := A^T\left[A\nabla^2F(\bar{x}_{\bar{\delta}})^{-1}A^T\right]^{-1}A$ is symmetric positive semidefinite,
applying \eqref{eq:s4_proof1_est3} with $Q = H_{\bar{\delta}}$, $u = \bar{x}_{\bar{\delta}1} - \bar{x}_{\bar{\delta}}$ and $v = \bar{x}_{\bar{\delta}} -
\hat{x}$, we have
\begin{eqnarray}\label{eq:s4_proof1_est4}
\bar{\lambda}_1^2 \leq \frac{1}{(1 \!-\! \Delta)^2}\Big\{ \!\left[(\bar{x}_{\bar{\delta}1} \!-\!
\bar{x}_{\bar{\delta}})^T\!H_{\bar{\delta}}(\bar{x}_{\bar{\delta}1} \!-\! \bar{x}_{\bar{\delta}})\right]^{1/2}  \!+\! \left[(\bar{x}_{\bar{\delta}} \!-\!
\hat{x})^T\!H_{\bar{\delta}}(\bar{x}_{\bar{\delta}} \!-\! \hat{x})\right]^{1/2} \!\!\Big\}^2.
\end{eqnarray}
Note that $H_{\bar{\delta}} \preceq \nabla^2F(\bar{x}_{\bar{\delta}})$ due to \eqref{eq:s4_basic_est_on_Hessian_of_F_3}. The first term of the right-hand side
of
\eqref{eq:s4_proof1_est4} satisfies
\begin{equation}\label{eq:s4_proof1_est5}
[\cdots] \leq  (\bar{x}_{\bar{\delta}+} - \bar{x}_{\bar{\delta}})^T\nabla^2F(\bar{x}_{\bar{\delta}})(\bar{x}_{\bar{\delta}1} - \bar{x}_{\bar{\delta}}) =
\Delta^2.
\end{equation}
On the other hand, by substituting $\bar{x}_{\bar{\delta}1}$ by $\bar{x}_{\bar{\delta}}$ into \eqref{eq:s4_proof1_est1}, we get
\begin{eqnarray}\label{eq:s4_proof1_est6}
\bar{\lambda}^2 =  (\bar{x}_{\bar{\delta}} \!-\! \hat{x})^TA^T\!\left[A\nabla^2F(\bar{x}_{\bar{\delta}})^{-1}A^T\right]^{-1}A(\bar{x}_{\bar{\delta}} \!-\!
\hat{x}) \!=\! (\bar{x}_{\bar{\delta}} \!-\! \hat{x})^TH_{\bar{\delta}}(\bar{x}_{\bar{\delta}} \!-\! \hat{x}).
\end{eqnarray}
Combining \eqref{eq:s4_proof1_est4}, \eqref{eq:s4_proof1_est5} and \eqref{eq:s4_proof1_est6}, we obtain
\begin{equation*}
\bar{\lambda}_1^2 \leq \frac{(\Delta + \bar{\lambda})^2}{(1-\Delta)^2},
\end{equation*}
which is equivalent to \eqref{eq:s4_basic_est_on_NT_decrements}.
\end{proof}

\paragraph{The proof of Lemma \ref{le:s4_main_inequalities}}
Since $\delta_1 + 2\Delta + \bar{\lambda} < 1$, it implies that $\delta_1 < 1$, $\Delta < 1/2$ and $\bar{\lambda} < 1$.
The proof of Lemma \ref{le:s4_main_inequalities} is divided into several steps as follows.

\vskip0.1cm
\noindent\textit{Step 1.}
First, we prove the following inequality:
\begin{align}\label{eq:s4_proof2_step1}
\bar{\lambda}_{+} \leq \frac{1}{(1-\delta_{+})}\left\{\delta_{+} +  \frac{1}{(1-\norm{p}_y)}\left[\delta_1 + \frac{(2\delta_1 - \delta_1^2)}{(1 -
\delta_1)^2}\norm{p}_y +
\frac{\norm{p}_y^2}{1-\norm{p}_y}
\right]\right\},
\end{align}
where $p := y_{+} - y$.
Indeed, it follows \eqref{eq:s4_basic_est_on_Hessian_of_dual_fun} that
\begin{align}\label{eq:s4_proof2_est2}
\bar{\lambda}_{+} &= \dnorm{\nabla\tilde{d}_{\bar{\delta}}(y_{+},t_{+})}_{y_{+}}^{*} =
\left[\nabla\tilde{d}_{\bar{\delta}}(y_{+},t_{+})\nabla^2\tilde{d}_{\bar{\delta}}(y_{+},t_{+})^{-1}\nabla\tilde{d}_{\bar{\delta}}(y_{+},t_{+})\right]^{1/2}
\nonumber\\
& \overset{\tiny\eqref{eq:s4_basic_est_on_Hessian_of_dual_fun}}{\leq}
\frac{1}{(1-\delta_{+})}\left[\nabla\tilde{d}_{\bar{\delta}}(y_{+},t_{+})\nabla^2\tilde{d}(y_{+},
t_{+})^{-1}\nabla\tilde{d}_{\bar{\delta}}(y_{+},t_{+})\right]^{1/2} \\
&\leq \frac{1}{(1-\delta_{+})}\norm{\nabla\tilde{d}_{\bar{\delta}}(y_{+},t_{+})}_{y_{+}}^{*}. \nonumber
\end{align}
Next, using \eqref{eq:s4_basic_est_on_diff_grads_of_dual_fun} we have
\begin{eqnarray}\label{eq:s4_proof2_est3}
\norm{\nabla\tilde{d}_{\bar{\delta}}(y_{+},t_{+})}_{y_{+}}^{*} &&\leq \norm{\nabla\tilde{d}(y_{+},t_{+})}_{y_{+}}^{*}  +
\norm{\nabla\tilde{d}_{\bar{\delta}}(y_{+},t_{+}) -
\nabla\tilde{d}(y_{+},t_{+})}_{y_{+}}^{*}  \nonumber\\
[-1.5ex]\\[-1.5ex]
&& \overset{\tiny\eqref{eq:s4_basic_est_on_diff_grads_of_dual_fun}}{\leq} \norm{\nabla\tilde{d}(y_{+},t_{+})}_{y_{+}}^{*} + \delta_{+}. \nonumber
\end{eqnarray}
Since $\tilde{d}(\cdot, t_{+})$ is standard self-concordant according to Lemma \ref{le:s3_self_concordancy}, one has
\begin{eqnarray}\label{eq:s4_proof2_est4}
\norm{\nabla\tilde{d}(y_{+},t_{+})}_{y_{+}}^{*} &&\leq \frac{1}{1-\norm{y_{+}-y}_y}\norm{\nabla\tilde{d}(y_{+},t_{+})}_{y}^{*} \nonumber \\
[-1.5ex]\\[-1.5ex]
&& = \frac{1}{1-\norm{p}_y}\norm{\nabla\tilde{d}(y_{+},t_{+})}_{y}^{*}. \nonumber
\end{eqnarray}
Plugging \eqref{eq:s4_proof2_est4} and \eqref{eq:s4_proof2_est3} into \eqref{eq:s4_proof2_est2} we obtain
\begin{equation}\label{eq:s4_proof2_est5}
\bar{\lambda}_{+} \leq \frac{1}{(1-\delta_{+})}\left[\frac{\norm{\nabla\tilde{d}(y_{+}, t_{+})}_{y}^{*}}{1-\norm{p}_y} + \delta_{+}\right].
\end{equation}
On the other hand, from \eqref{eq:s4_inexact_fs_NT_invariant}, we have
\begin{eqnarray}\label{eq:s4_proof2_est6}
\nabla\tilde{d}(y_{+},t_{+}) &&\overset{\tiny\eqref{eq:s4_inexact_fs_NT_invariant}}{=} \nabla\tilde{d}(y_{+}, t_{+}) - \left[\nabla\tilde{d}_{\bar{\delta}}(y,
t_{+}) +
\nabla^2\tilde{d}_{\bar{\delta}}(y,t_{+})(y_{+}-y)\right] \nonumber\\
&& = \left[\nabla\tilde{d}(y,t_{+}) - \nabla\tilde{d}_{\bar{\delta}}(y,t_{+}) \right] \nonumber\\
[-1.5ex]\\[-1.5ex]
&& + \left\{ [\nabla^2\tilde{d}(y,t_{+})-\nabla^2\tilde{d}_{\bar{\delta}}(y,t_{+})](y_{+}-y) \right\} \nonumber\\
&& + \left[\nabla\tilde{d}(y_{+},t_{+}) - \nabla\tilde{d}(y,t_{+}) - \nabla^2\tilde{d}(y,t_{+})(y_{+}-y)\right]. \nonumber
\end{eqnarray}
By substituting $t$ by $t_{+}$ in \eqref{eq:s4_basic_est_on_diff_grads_of_dual_fun}, we obtain an estimate of the first term of \eqref{eq:s4_proof2_est6} as
\begin{align}\label{eq:s4_proof2_est7}
\norm{\nabla\tilde{d}(y,t_{+}) - \nabla\tilde{d}_{\bar{\delta}}(y, t_{+})}_{y}^{*} \leq \norm{\bar{x}_{\bar{\delta}1} - x^{*}_1}_{x^{*}_1} = \delta_1.
\end{align}
Next, we consider the second term of \eqref{eq:s4_proof2_est6}. It follows from \eqref{eq:s4_basic_est_on_Hessian_of_dual_fun} that
\begin{eqnarray}\label{eq:s4_proof2_est8}
\left[(1-\delta_1)^2 - 1\right]\nabla^2\tilde{d}(y,t_{+}) &&\preceq \nabla^2\tilde{d}_{\bar{\delta}}(y,t_{+}) - \nabla^2\tilde{d}(y,t_{+}) \nonumber\\
[-1.5ex]\\[-1.5ex]
&&\preceq \left[(1-\delta_1)^{-2} - 1\right]\nabla^2\tilde{d}(y,t_{+}). \nonumber
\end{eqnarray}
If we define $G := \left[\nabla^2\tilde{d}_{\bar{\delta}}(y,t_{+}) - \nabla^2\tilde{d}(y,t_{+})\right]$ and $H :=
\nabla^2\tilde{d}(y,t_{+})^{-1/2}G\nabla^2\tilde{d}(y,t_{+})^{-1/2}$
then
\begin{align}\label{eq:s4_proof2_est9}
\norm{[\nabla^2\tilde{d}(y,t)-\nabla^2\tilde{d}_{\bar{\delta}}(y,t_{+})](y_{+}-y)}_{y}^{*} & = \norm{Gp}_y^{*} \leq \norm{H}\norm{p}_{y},
\end{align}
where, by virtue of \eqref{eq:s4_proof2_est8} and the condition $\delta_1 < 1$, one has
\begin{equation*}
\norm{H} \leq \max\left\{1-(1-\delta_1)^2, \frac{1}{(1-\delta_1)^2}-1\right\} = \frac{2\delta_1 - \delta_1^2}{(1-\delta_1)^2}.
\end{equation*}
Hence, \eqref{eq:s4_proof2_est9} leads to
\begin{align}\label{eq:s4_proof2_est10}
\norm{[\nabla^2\tilde{d}(y,t)-\nabla^2\tilde{d}_{\bar{\delta}}(y,t_{+})](y_{+}-y)}_{y}^{*} \leq \frac{(2\delta_1 -\delta_1^2)}{(1-\delta_1)^2}\norm{p}_{y}.
\end{align}
Furthermore, since $\tilde{d}(\cdot, t)$ is standard self-concordant, similar to the proof of \cite[Theorem 4.1.14]{Nesterov2004}, we have
\begin{align}\label{eq:s4_proof2_est11}
\norm{\nabla\tilde{d}(y_{+},t_{+}) - \nabla\tilde{d}(y,t_{+}) - \nabla^2\tilde{d}(y,t_{+})(y_{+}-y)}_{y}^{*} \leq \frac{\norm{p}_y^2}{1-\norm{p}_y}.
\end{align}
Now, we apply the triangle inequality $\norm{a+b+c}^{*}_y \leq \norm{a}_y^{*} + \norm{b}_y^{*} + \norm{c}_y^{*}$ to \eqref{eq:s4_proof2_est6} and then plugging
\eqref{eq:s4_proof2_est7},
\eqref{eq:s4_proof2_est10} and \eqref{eq:s4_proof2_est11} into the resulted inequality to obtain
\begin{align*}
\norm{\nabla\tilde{d}_{\bar{\delta}}(y_{+},t_{+})}_y^{*} &\leq \delta_1 + \frac{(2\delta_1 - \delta_1^2)}{(1-\delta_1)^2}\norm{p}_y +
\frac{\norm{p}_y^2}{1-\norm{p}_y}.
\end{align*}
Finally, by substituting this inequality into \eqref{eq:s4_proof2_est5} we get \eqref{eq:s4_proof2_step1}.

\vskip0.1cm
\noindent\textit{Step 2.} Next, we estimate \eqref{eq:s4_proof2_step1} in terms of $\bar{\lambda}_1$ to obtain
\begin{eqnarray}\label{eq:s4_proof2_est1b}
\bar{\lambda}_{+} \leq \frac{1}{(1 \!-\! \delta_{+})}\!\left[\!\left(\!\frac{\bar{\lambda}_1}{1 \!-\! \delta_1 \!-\! \bar{\lambda}_1}\!\right)^2 \!+\!
\frac{(2\delta_1 \!-\! \delta^2)}{(1 \!-\! \delta_1)^2}\!\!\left(\!\frac{\bar{\lambda}_1}{1 \!-\! \delta_1 \!-\! \bar{\lambda}_1}\!\right)
\!+\! \frac{(1 \!-\! \delta_1)\delta_1}{1 \!-\! \delta_1 \!-\! \bar{\lambda}_1} \!+\!
\delta_{+} \! \right].
\end{eqnarray}
Indeed, by using \eqref{eq:s4_basic_est_on_Hessian_of_F} with $x = \bar{x}_{\bar{\delta}1}$ and $z = x^{*}_1$ and then \eqref{eq:s3_deriv_of_sm_dual_func} we
have
\begin{equation*}
(1-\delta_1)^2\nabla^2\tilde{d}_{\bar{\delta}}(y, t_{+}) \preceq \nabla^2\tilde{d}(y, t_{+}) \preceq (1-\delta_1)^{-2}\nabla^2\tilde{d}_{\bar{\delta}}(y,
t_{+}).
\end{equation*}
This inequality together with the definition of $\dnorm{\cdot}$ imply
\begin{align*}
(1-\delta_1)\dnorm{p}_y \leq \norm{p}_y = \left[p^T\nabla^2d(y,t_{+})p\right]^{1/2} \leq (1-\delta_1)^{-1}\dnorm{p}_y.
\end{align*}
Moreover, since $\dnorm{p}_y = \dnorm{\nabla\tilde{d}_{\bar{\delta}}(y, t_{+})}^{*}_y = \bar{\lambda}_1$ due to \eqref{eq:s4_inexact_fs_NT_invariant}, the last
inequality is equivalent to
\begin{equation}\label{eq:s4_proof2_est12}
\norm{p}_y \leq \frac{\bar{\lambda}_1}{1-\delta_1}.
\end{equation}
Note that the right-hand side of \eqref{eq:s4_proof2_step1} is nondecreasing w.r.t. $\norm{p}_y$ in $[0,1)$. Substituting \eqref{eq:s4_proof2_est12}
into \eqref{eq:s4_proof2_step1} we finally obtain \eqref{eq:s4_proof2_est1b}.

\vskip0.1cm
\noindent\textit{Step 3.} We further estimate \eqref{eq:s4_proof2_est1b} in terms of $\Delta$ and $\bar{\lambda}$.
First, we can easily check that the right-hand side of \eqref{eq:s4_proof2_est1b} is nondecreasing with respect to $\bar{\lambda}_1$, $\delta_1$ and
$\delta_{+}$.
Now, by using the definitions of $\Delta$ and $\bar{\lambda}$, it follows from Lemma \ref{le:s4_basic_estimates} c) that
\begin{equation*}\label{eq:s4_proof2_est16a}
\bar{\lambda}_1 \leq \frac{\bar{\lambda} + \Delta}{1-\Delta}.
\end{equation*}
Since $\delta_{+} < 1$ and $\delta_1 + 2\Delta + \bar{\lambda} < 1$, substituting this inequality into \eqref{eq:s4_proof2_est1b}, we obtain
\begin{eqnarray}\label{eq:s4_proof2_est15}
\bar{\lambda}_{+} &&\leq \frac{1}{(1 - \delta_{+})}\left[\delta_{+} + \left(\frac{\bar{\lambda} + \Delta}{1 - \delta_1 - 2\Delta - \bar{\lambda}}\right)^2 +
\frac{(2\delta_1 - \delta_1^2)}{(1-\delta_1)^2}\left(\frac{\bar{\lambda} + \Delta}{1-\delta_1 - 2\Delta - \bar{\lambda}}\right) \right. \nonumber\\
[-1.5ex]\\[-1.5ex]
&&\left. + \frac{\delta_1(1 - \delta_1)(1-\Delta)}{1-\delta_1 - 2\Delta - \bar{\lambda}}\right]. \nonumber
\end{eqnarray}
The right-hand side of \eqref{eq:s4_proof2_est15} is well-defined and also nondecreasing with respect to all variables.

\vskip0.1cm
\noindent\textit{Step 4.}
Finally, we facilitate the right-hand side of \eqref{eq:s4_proof2_est15} to obtain \eqref{eq:s4_main_estimate}.
Since $\bar{\lambda} \geq 0$, we have
\begin{align*}
(1-\delta_1)(1-\Delta) &= 1 - \Delta - \delta_1 + \delta_1\Delta  = [1 - \delta_1 - 2\Delta - \bar{\lambda}] + (\bar{\lambda} + \Delta) + \delta_1\Delta\\
&\leq   [1 - \delta_1 - 2\Delta - \bar{\lambda}] + (\bar{\lambda}_1 + \Delta) + \delta_1(\Delta + \bar{\lambda})\\
& = [1 - \delta_1 - 2\Delta - \bar{\lambda}] + (1 + \delta_1)(\bar{\lambda} + \Delta).
\end{align*}
Therefore, this inequality implies
\begin{equation}\label{eq:s4_proof2_est16}
\frac{\delta_1(1-\delta_1)(1-\Delta)}{1-\delta_1 - 2\Delta - \bar{\lambda}} \leq \delta_1 + \delta_1(1 + \delta_1)\left[\frac{\Delta + \bar{\lambda}}{1 -
\delta_1 - 2\Delta - \bar{\lambda}}\right].
\end{equation}
Alternatively, since $0\leq \delta_1 < 1$, we have $1 + \delta_1 \leq \frac{1}{1-\delta_1}$. Thus
\begin{align*}
\frac{2\delta_1 - \delta_1^2)}{(1-\delta_1)^2} + \delta_1(1 + \delta_1) &= \delta_1\left[\frac{1}{(1-\delta_1)^2} + \frac{1}{(1-\delta_1)} + (1 + \delta_1)
\right] \\
& \leq \delta_1\left[\frac{1}{(1-\delta_1)^2} + \frac{2}{1-\delta_1}\right].
\end{align*}
Substituting  inequality \eqref{eq:s4_proof2_est16} into \eqref{eq:s4_proof2_est15} and then using the last inequality and $\xi :=
\frac{\bar{\lambda}+\Delta}{1-\delta_1 - 2\Delta - \bar{\lambda}}$, we obtain  \eqref{eq:s4_main_estimate}.

\vskip0.1cm
\noindent\textit{Step 5.}
The nondecrease of the right-hand side of \eqref{eq:s4_main_estimate} is obvious.
The inequality \eqref{eq:s4_main_estimate_for_exact_case} follows directly from \eqref{eq:s4_main_estimate} by noting that $\bar{\lambda} \equiv \lambda$ and
$\bar{x}_{\bar{\delta}} \equiv
x^{*}$.
\Eproof

\section{Path-following decomposition algorithm with exact Newton iterations}\label{sec:path_following_for_exact_case}
In Algorithm \ref{as:A1}, if we set $\bar{\delta} = 0$, then this algorithm collapses to the ones considered  in
\cite{Kojima1993,Necoara2009,Shida2008,Zhao2001,Zhao2005}. However, we emphasize the following points.
\begin{enumerate}
\item We consider this variant as a special case of the algorithm presented in the previous sections which is called \textit{path-following decomposition
algorithm with exact Newton iterations}.
\item In \cite{Kojima1993,Necoara2009,Shida2008,Zhao2001,Zhao2005}, since the primal subproblem \eqref{eq:s3_smoothed_dual_func} is solved exactly, the family
$\{d(\cdot, t)\}_{t > 0}$ of the smooth dual functions is strongly self-concordant due to Legendre transformation.
Consequently, the standard theory of interior point methods in \cite{Nesterov1994} can be applied to minimize this function.
In contrast to those, in this paper we analyze directly the path-following iterations to select appropriate parameters for implementation.
\end{enumerate}
Note that the radius of the neighbourhood of the analytic center in Algorithm \ref{alg:B1} below is $\beta^{*} = \frac{1}{2}(3-\sqrt{5}) \approx 0.381966$
compared to the one used in literature,
$\beta^{*} = 2-\sqrt{3} \approx 0.26795$.

\subsection{Analyzing the exact path-following iteration}\label{subsec:s5_analyzing_exact_PPIter}
Let us assume that the primal subproblem \eqref{eq:s3_smoothed_dual_func} is solved exactly, i.e. $\bar{\delta} = 0$.
Then, we have $\bar{x}_{\bar{\delta}} \equiv x^{*}$ and $\delta(\bar{x}_{\bar{\delta}}, x^{*}) = 0$ for all $y\in Y$ and $t > 0$. Moreover, it follows from
\eqref{eq:s4_choice_of_Delta_star} that $\Delta = \Delta^{*} = \norm{x^{*}(y, t_{+}) - x^{*}(y,t)}_{x^{*}(y,t)}$.
We consider one step of the path-following scheme with exact full-step Newton iterations:
\begin{equation}\label{eq:s6_exact_path_following_iter}
\begin{cases}
t_{+} := t - \Delta t, ~\Delta t > 0,\\
y_{+} := y - \nabla^2d(y,t_{+})^{-1}\nabla{d}(y, t_{+}) \equiv y - \nabla^2\tilde{d}(y,t_{+})^{-1}\nabla\tilde{d}(y, t_{+}).
\end{cases}
\end{equation}
For sake of notation simplicity, we denote by $\tilde{\lambda} := \lambda_{\tilde{d}(\cdot, t)}(y)$, $\tilde{\lambda}_1 := \lambda_{\tilde{d}(\cdot, t_{+})}(y)$
and $\tilde{\lambda}_{+} :=
\lambda_{\tilde{d}(\cdot, t_{+})}(y_{+})$. It follows from \eqref{eq:s4_main_estimate_for_exact_case} of Lemma \ref{le:s4_main_inequalities} that
\begin{equation}\label{eq:s6_exact_main_estimate}
\tilde{\lambda}_{+} \leq \left(\frac{\tilde{\lambda} + \Delta^{*}}{1 - 2\Delta^{*} -\tilde{\lambda}}\right)^2.
\end{equation}
Now, we fix $\beta \in (0, 1)$ such that $\tilde{\lambda} \leq \beta$. We need to find a condition on $\Delta$ such that $\tilde{\lambda}_{+} \leq \beta$.
Indeed, since the right-hand side of \eqref{eq:s6_exact_main_estimate} is nondecreasing with respect to $\tilde{\lambda}$, it implies that
$\tilde{\lambda}_{+} \leq \left(\frac{\Delta^{*} + \beta}{1-2\Delta^{*} - \beta}\right)^2$.
Thus $\tilde{\lambda}_{+} \leq \beta$ if $\frac{\Delta^{*} + \beta}{1-2\Delta^{*} - \beta} \leq \sqrt{\beta}$ which leads to
\begin{equation}\label{eq:s6_exact_choice_of_Delta}
0\leq \Delta^{*} \leq \bar{\Delta}^{*} := \frac{\sqrt{\beta}(1-\sqrt{\beta} - \beta)}{1+2\sqrt{\beta}},
\end{equation}
provided that
\begin{equation}\label{eq:s6_exact_choice_of_beta}
0 < \beta < \beta^{*} := \frac{3-\sqrt{5}}{2} \approx 0.381966.
\end{equation}
Since $\Delta \equiv \Delta^{*}$, according to \eqref{eq:s4_bar_par_decrement}, we can choose
\begin{equation}\label{eq:s6_exact_bar_par_decrement}
\Delta t := \sigma t = \frac{\bar{\Delta}^{*}t}{\sqrt{\nu} + (\sqrt{\nu} + 1)\bar{\Delta}^{*}},
\end{equation}
where $\sigma := \frac{\bar{\Delta}^{*}}{\sqrt{\nu} + (\sqrt{\nu} + 1)\bar{\Delta}^{*}}$. Therefore, $t$ is updated by $t_{+} := t - \Delta t = (1-\sigma)t$.
Note that $t$ decreases linearly with the contraction factor $(1-\sigma)$.

In particular, if we choose $\beta = \frac{\beta^{*}}{4} \approx 0.095492$ then $\bar{\Delta}^{*} \approx 0.113729$,
which leads to $(1-\sigma) = \frac{\sqrt{\nu}(\bar{\Delta}^{*}+1)}{\sqrt{\nu}(\bar{\Delta}^{*} + 1) + \bar{\Delta}^{*}} \approx
\frac{1.1137\sqrt{\nu}}{1.1137\sqrt{\nu} + 0.1137}$.

\subsection{The algorithm and its convergence}\label{subsec:exact_case_Phase2}
Let us fix an initial value $t = t_0 > 0$ and $\beta \in (0, \beta^{*})$, where $\beta^{*}$ is given in \eqref{eq:s6_exact_choice_of_beta}.
First, we apply Phase 1 to find a starting point $y^0 \in Y$ such that $\tilde{\lambda}_0 := \lambda_{\tilde{d}(\cdot, t_0)}(y^0) \leq \beta$.
This phase is carried out by applying the damped Newton iteration scheme proposed in \cite{Nesterov2004}.
Then we perform the path-following algorithm.
From Definition \ref{de:s3_epsilon_dual_sol}, we can see that if $t_k \leq \frac{\varepsilon_d}{\omega_{*}(\beta)}$ then $y^k$ is a  $2\varepsilon_d$-solution
of  \eqref{eq:s2_original_dual_prob}.
The algorithm is presented in detail as follows.

\noindent\rule[1pt]{\textwidth}{1.0pt}{~~}
\begin{algorithm}\label{alg:B1}{~}$($\textit{Path-following algorithm with exact Newton iterations}$)$
\end{algorithm}
\vskip -0.2cm
\noindent\rule[1pt]{\textwidth}{0.5pt}
\noindent\textbf{Initialization:} Perform the following steps:
\begin{enumerate}
\item Fix a constant $\beta\in (0, \beta^{*})$ (e.g. $\beta = \frac{1}{4}\beta^{*}$), where $\beta^{*} = \frac{3-\sqrt{5}}{2} \approx 0.381966$.
\item Compute $\bar{\Delta} := \frac{\sqrt{\beta}(1-\sqrt{\beta}-\beta}{1+2\sqrt{\beta}}$ and $\sigma := \frac{\bar{\Delta}}{\sqrt{\nu}+(\sqrt{\nu} +
1)\bar{\Delta}}$.
\item Fix a tolerance $\varepsilon_d > 0$ and choose an initial value $t_0 > 0$. 
\end{enumerate}
\noindent\textbf{Phase 1. }(\textit{Finding a starting point}).
\begin{enumerate}
\item Choose an arbitrary starting point $y^{0,0} \in Y$.
\end{enumerate}
\texttt{For} $j = 0,1,\cdots$ \texttt{do}
\begin{enumerate}
\item Solve exactly the primal subproblem \eqref{eq:s3_smoothed_dual_func} \textit{in parallel} to obtain $x^{*}(y^{0,j},t_0)$.
\item Evaluate $\nabla{d}(y^{0,j},t_0)$ and $\nabla{d}(y^{0,j}, t_0)$ by \eqref{eq:s3_deriv_of_sm_dual_func} and \eqref{eq:s3_deriv_of_sm_dual_func},
respectively.
\item Compute the Newton decrement $\tilde{\lambda}_j = \lambda_{\tilde{d}(\cdot, t_0)}(y^{0,j})$.
\item If $\tilde{\lambda}_j \leq \beta$ then set $y^0 := y^{0,j}$ and terminate.
\item Update $y^{0,j+1}$ as $y^{0,j+1} := y^{0,j} - \alpha_j\nabla^2{d}(y^{0,j}, t_0)^{-1}\nabla{d}(y^{0,j}, t_0)$,
where the step size $\alpha_j := \frac{1}{1+\tilde{\lambda}_j} \in (0, 1]$.
\end{enumerate}
\texttt{End of For}.\\
\noindent\textbf{Phase 2.} (\textit{Path-following iterations}).\\
\texttt{For} $k=0,1,\cdots$ \texttt{do}
\begin{enumerate}
\item If $t_k \leq \frac{\varepsilon_d}{\omega_{*}(\beta)}$ then terminate.
\item Update $t_k$ as $t_{k+1} := (1-\sigma)t_k$.
\item Solve exactly the primal subproblem \eqref{eq:s3_smoothed_dual_func} \textit{in parallel} to obtain a solution $x^{*}(y^k,t_{k+1})$.
\item Evaluate $\nabla{d}(y^k,t_{k+1})$ and $\nabla{d}(y^k, t_{k+1})$ by \eqref{eq:s3_deriv_of_sm_dual_func} and \eqref{eq:s3_deriv_of_sm_dual_func},
respectively.
\item Update $y^{k+1}$ as $y^{k+1} := y^k + \Delta y^k = y^k - \nabla^2{d}(y^k, t_{k+1})^{-1}\nabla{d}(y^k, t_{k+1})$.
\end{enumerate}
\texttt{End of For}.\\
\vskip-0.2cm
\noindent\rule[1pt]{\textwidth}{1.0pt}

As in Algorithm \ref{alg:A1}, the main task of this algorithm is Step 1 in Phase 1 and Step 3 in Phase 2, which can be carried out in parallel, and Step 5
in Phase 1 and Step 4 in Phase 2, which require  a centralized computation to solve the linear system $\nabla^2d(y^k, t_{k+1})\Delta y = -\nabla{d}(y^k,
t_{k+1})$ (see Section \ref{sec:implementation_detail}). In an implementation, the primal subproblem can not be solved exactly but it must be solved up to a
very high accuracy.

Since $\tilde{d}(\cdot, t_0)$ is standard self-concordant due to Lemma \ref{le:s3_self_concordancy}. By \cite[Theorem 4.1.12]{Nesterov2004}, the number of
iterations to obtain $y^0\in Y$
such that $\lambda_{\tilde{d}(\cdot, t_0)}(y^0)\leq\beta$ does not exceed
\begin{equation}\label{eq:s6_maxiter_of_Phase1}
\bar{J}_{\max} := \left\lfloor\frac{\tilde{d}(y^{0,0}, t_0) - \tilde{d}^{*}(t_0)}{\omega(\beta)}\right\rfloor + 1 = \left\lfloor\frac{d(y^{0,0}, t_0) -
d^{*}(t_0)}{t_0\omega(\beta)}\right\rfloor + 1.
\end{equation}
The number $\bar{J}_{\max}$ not only depends on the distance $d(y^{0,0}, t_0) - d^{*}(t_0)$ but also on $t_0$.
If we choose $t_0$ small then $\bar{J}_{\max}$ is large, while the number of iterations in Algorithm \ref{alg:B1} is small. Therefore, in the
implementation, we need to balance between these quantities to get a good performance.

The convergence of Phase 2 in Algorithm \ref{alg:B1} is stated in the following theorem.

\begin{theorem}\label{th:s6_exact_PP_convergence}
Let $t_0 > 0$ and $y^0\in Y$ such that $\lambda_{\tilde{d}(\cdot, t_0)}(y^0) \leq \beta$. Then the maximum number of iterations $k$ needed by Algorithm
\ref{alg:B1} to obtain a
$2\varepsilon_d$ - solution $y^k$ of \eqref{eq:s2_original_dual_prob} does not exceed
\begin{equation}\label{eq:s6_exact_PP_max_iter}
\bar{k} := \left\lfloor\frac{\ln\left(\frac{t_0\omega_{*}(\beta)}{\varepsilon_d}\right)}{\ln\left(1+\frac{\bar{\Delta}^{*}}{\sqrt{\nu}(\bar{\Delta}^{*} +
1)}\right)}\right\rfloor + 1,
\end{equation}
where $\bar{\Delta}^{*}$ is defined by \eqref{eq:s6_exact_choice_of_Delta}.
\end{theorem}

\begin{proof}
From Step 2 of Algorithm \ref{alg:B1}, we have $t_k = (1-\sigma)^kt_0 = \left(1 + \frac{\bar{\Delta}^{*}}{\sqrt{\nu}(\bar{\Delta}^{*}+1)}\right)^kt_0$.
Algorithm
\ref{alg:B1} is terminated if $t_k \leq \frac{\varepsilon_d}{\omega_{*}(\beta)}$. Thus $\left(1 + \frac{\bar{\Delta}^{*}}{\sqrt{\nu}(\bar{\Delta}^{*} +
1)}\right)^k \leq
\frac{\varepsilon_d}{t_0\omega_{*}(\beta)}$, which leads to \eqref{eq:s6_exact_PP_max_iter}.
\end{proof}

\begin{remark}[\textbf{The worst-case complexity}]\label{re:s6_complexity}
Since $\ln\left(1+\frac{\bar{\Delta}^{*}}{\sqrt{\nu}(\bar{\Delta}^{*} + 1)}\right) \sim \frac{\bar{\Delta}^{*}}{\sqrt{\nu}(\bar{\Delta}^{*} + 1)}$,
the worst-case
complexity of Algorithm \ref{alg:B1} is $O(\sqrt{\nu}\ln(t_0/\varepsilon_d))$.
\end{remark}

\begin{remark}[\textbf{Damped Newton iteration}]\label{re:s5_damped_NT_path_following}
Note that, at Step 5 of Algorithm \ref{alg:B1}, we can use a damped Newton iteration $y^{k+1} := y^k - \alpha_k\nabla^2{d}(y^k, t_{k+1})^{-1}\nabla{d}(y^k,
t_{k+1})$
instead of the full-step Newton iteration, where $\alpha_k = (1 + \lambda_{\tilde{d}(\cdot, t_{k+1})}(y^k))^{-1}$. In this case, with the same argument as
before, we can compute $\beta^{*} = 0.5$ and $\Delta^{*} = \frac{\sqrt{0.5\beta} -
\beta}{1+\sqrt{0.5\beta}}$.
\end{remark}

\section{Discussion on implementation}\label{sec:implementation_detail}
In this section, we first show how to handle a general concave objective function. Next, we discuss on solving the primal subproblem
\eqref{eq:s3_smoothed_dual_func} including local equality constraints. Finally, we briefly describe a parallel method to compute the Newton-type direction
for the master problem.

\subsection{Handling general objective function}\label{subsec:handling_general_objective_func}
If $\phi_i$ is nonlinear, concave and its epi-graph is endowed with a self-concordant barrier for some $i\in I_M := \{1,\dots, M\}$, then we propose
to use slack variables to move the objective function into constraints.
Let us denote by $\hat{x}_{i} := (x_{i}^T, s_i)^T$ and
\begin{equation*}
\hat{X}_{i} := \left\{(x_i, s_i) ~|~ x_i \in X_i, ~s_i \geq \underline{s}_i, ~\phi_i(x_i) \geq s_i\right\},
\end{equation*}
for a sufficiently small value $\underline{s}_i$ such that the constraint $s_i \geq \underline{s}_i$ is inactive. Let $\hat{F}_i$ be a self-concordant barrier
of $\hat{X}_i$ and let $\hat{c}_i := (0^T, 1)^T \in \mathbb{R}^{n_i+1}$. Then problem \eqref{eq:s1_main_CP} can be transformed into a convex separable
optimization problem with linear objective function. In this case, the algorithms developed in the previous sections can be applied to solve the resulting
problem.

If $\phi_i$ is concave quadratic then, according to \cite[Theorem 3.3.1]{Nesterov1994}, we can construct a self-concordant barrier $G_i(\hat{x}_{i}) :=
-\ln(\phi_i(x_{i}) - s_i)$ for the epi-graph of $\phi_i$. Particularly, the optimality condition for this problem is $ \hat{c} + \hat{A}^Ty -
t\nabla\hat{F}(\hat{x}) = 0$, which can be written as
\begin{equation*}
\begin{cases}
&A^Ty - t\nabla F(x) - t\mathrm{diag}(f_i(x_{i}) - s_i)^{-1}\nabla f(x) = 0,\\
&t\mathrm{diag}(f_i(x_{i})-s_i)^{-1} = \mathsf{1}.
\end{cases}
\end{equation*}
By substituting the second line into the first line of the above expression, we obtain
\begin{equation*}
A^Ty - t\nabla F(x) + \nabla f(x) = 0.
\end{equation*}
However, this condition is indeed the optimality condition of the following problem
\begin{eqnarray}\label{eq:extra_smoothed_dual_func}
d(y, t) := \max_{x\in \int(X)}\left\{ f(x) + y^T(Ax-b) - t[F(x) - F(x^c)]\right\}.
\end{eqnarray}
Consequently, the algorithms developed in the previous sections can be applied to solve \eqref{eq:s1_main_CP} without moving $\phi_i$ into the constraints.

Several examples of convex problems for which the logarithmic function $G_i(\hat{x}_{i})$ is self-concordant can be found in \cite{Hertog1992}.
Note that, in some problems, we may need to reformulate the epi-graph of $f_i$ to obtain a self-concordant barrier.
For example, many optimization problems in network use an objective function of the form $\phi_i(x_i) = \frac{x_i}{1 - x_i}$, where $0 \leq x_i < 1$. The
inequality presented the epi-graph of $\phi_i$ is $\frac{x_i}{1-x_i} \leq s_i$, which is equivalent to $\sqrt{(x_i + s_i)^2 + 4} \leq x_i - s_i - 2$. The last
inequality is indeed a second order cone constraint endowed with a $2$-self-concordant barrier \cite{Nesterov1994}.

\subsection{Solving the primal subproblems}\label{subsec:s7_solving_inner_primal_problem}
Let us recall the primal subproblem in \eqref{eq:extra_smoothed_dual_func} with a nonlinear objective function. We need to solve this problem inexactly up
to a desired accuracy $\varepsilon(t) > 0$, e.g. $\varepsilon(t) = \frac{\bar{\delta}t}{(\nu+2\sqrt{\nu})(1+\bar{\delta})}$.
Note that the approximate optimality condition of \eqref{eq:s3_smoothed_dual_prob} becomes
\begin{equation}\label{eq:s7_inner_primal_prob}
\norm{\nabla{f}(x) + A^Ty - t\nabla{F}(\bar{x})}_{x^c}^{*} \leq \varepsilon(t).
\end{equation}
By separability, this approximate problem can be solved in parallel as
\begin{equation}\label{eq:s7_inner_primal_prob_i}
\norm{\nabla{f}_i(x) + A^T_{i}y - t\nabla{F}_i(\bar{x}_{i})}_{x^c_{i}}^{*} \leq \varepsilon_i(t), ~\varepsilon_i(t) \geq 0, ~i=1,\dots, M,
\end{equation}
where $\sum_{i=1}^M\varepsilon_i(t) = \varepsilon(t)$. In principle, we can choose $\varepsilon_i(t) = \frac{\varepsilon(t)}{M}$.
However, in some practical situations, it is important to choose different $\varepsilon_i(t)$ for different components, especially, when some component problems
can be solved analytically in a closed form.

Since $F_i$ is standard self-concordant, the function $\psi_i(x_i; y, t):=  F_i(x_i) - t^{-1}(f_i(x_i)  +  y^TA_ix_i)$ is also standard self-concordant.
Moreover, $\nabla\psi_i(x_i; y, t) =  \nabla{F}_i(x_i) - t^{-1}(\nabla{f}_i(x_i) + A^T_iy)$ and $\nabla^2\psi_i(x_i; y, t) = \nabla^2F_i(x_i) -
\nabla^2f_i(x_i)$.
Since $\nabla^2\psi_i(x_i; y, t) \succ 0$, we define
\begin{equation*}
\lambda_{\psi_i}(x_i) := \left[\nabla\psi_i(x_i; y, t)\nabla\psi_i(x_i; y, t)^{-1}\nabla\psi_i(x_i; y, t)\right]^{-1/2},
\end{equation*}
the Newton decrement of $\psi_i$.

Now, let us apply Newton method to solve problem \eqref{eq:s7_inner_primal_prob}. First, we fix $\beta_i \in (0, \beta^{*})$, where $\beta^{*} :=
\frac{1}{2}(3-\sqrt{5})$, and choose $x^0_{i} \in \int(X_{i})$. Then, we generate a sequence $\{x^j_{i}\}_{j\geq 0}$ as
\begin{eqnarray}\label{eq:s7_newton_iter}
&&x^{j+1}_{i} := x^j_{i} + \alpha_{ij}\Delta{x}^j_{i}, \nonumber\\
[-1.2ex]\mathrm{where}~~&&\\[-1.2ex]
 &&\Delta{x}^j_{i} := - \nabla^2\psi_i(x^j_i; y, t)^{-1}\nabla\psi_i(x_i; y, t) ~\textrm{and} ~\alpha_{ij} \in (0, 1].\nonumber
\end{eqnarray}
Theoretically, the step-size $\alpha_{ij}$ can be chosen as $\alpha_{ij} := 1$ if $\lambda_{\psi_i}(x_i^j) \leq \beta_i$ and $\alpha_{ij} :=
(1+\lambda_{\psi_i}(x_i^j))^{-1}$, otherwise.
However, this choice is usually too conservative and not preferable in practice. Thus one can use an appropriate line-search procedure to select $\alpha_{ij}$.
Note that in linear programming, $F_i$ is diagonal, e.g. $F_i(x_{i}) = \textrm{diag}(-\ln(x_{i}))$, so that computing the Newton iteration
\eqref{eq:s7_newton_iter} requires a low computational cost. In general, we have to solve a linear system of the form 
\begin{equation*}
\nabla^2\psi_i(x^j_i; y, t)\Delta{x}_{i}^j = - \nabla\psi_i(x_i; y, t)
\end{equation*}
to obtain a Newton direction $\Delta{x}_{i}^j$.
The convergence of the Newton scheme \eqref{eq:s7_newton_iter} can be found in \cite{Nesterov2004}.
Note that in Algorithms \ref{alg:A1} and \ref{alg:A1a}, \eqref{eq:s7_inner_primal_prob_i} is solved repeatedly at different $t_k$. It is important to
warm-start the Newton iteration \eqref{eq:s7_newton_iter} by using the finally approximate solution of the previous iterate $t_{k-1}$ as a starting point for
the current one $t_k$.

Finally, if the local equality constraints $E_ix_i = f_i$ are available in \eqref{eq:s1_main_CP} for some $i\in\{1,\dots, M\}$, then the KKT condition of the
primal subproblem $i$ becomes
\begin{equation}\label{eq:s7_kkt_i_eq}
\begin{cases}c_i + A_i^Ty + E_i^Tz_i - t\nabla{F}_i(x_i) = 0, \\ E_ix_i - f_i = 0.\end{cases}
\end{equation}
Instead of the full KKT system \eqref{eq:s7_kkt_i_eq}, we consider a reduced KKT condition as follows
\begin{equation}\label{eq:s7_reduced_kkt_i_eq}
Z^T_i(c_i + A_i^Ty) - tZ_i^T\nabla{F}_i(Z_ix_i^z + R^{-T}_if_i) = 0.
\end{equation}
Here, $(Q_i, R_i)$ is a QR-factorization of $E_i^T$ and $[Y_i, Z_i] = Q_i$ is a basis of the range space and the null space of $E_i^T$, respectively. Due to the
invariance of the norm $\norm{\cdot}_{x^{*}}$, we can show that $\norm{\bar{x}_{\bar{\delta}}-x^{*}}_{x^{*}} = \norm{\bar{x}_{\bar{\delta}}^z -
x^{*z}}_{x^{*z}}$. Therefore, the condition \eqref{eq:s4_approx_sol} coincides with $\norm{\bar{x}_{\bar{\delta}}^z - x^{*z}}_{x^{*z}} \leq
\bar{\delta}$. However, the last condition is satisfied if 
\begin{equation}\label{eq:s7_condtion_for_approx_sol}
\norm{Z^T_i(c_i + A_i^Ty) - tZ_i^T\nabla{F}_i(Z_ix_i^z + R^{-T}_if_i)}_{x^{cz}_i}^{*} \leq \varepsilon_i,
\end{equation}
where $\sum_{i=1}^M\varepsilon_i = \varepsilon_p$ and $\varepsilon_p$ is defined by \eqref{eq:s4_choice_tol_for_inex_case}. Note that the QR-factorization of
$E_i^T$ can be computed one time, a priori.

\subsection{Computing the inexact perturbed Newton direction}\label{subsec:s7_compute_NT_search_direction}
Let us rewrite the inexact-perturbed Newton direction in Algorithms \ref{alg:A1} and \ref{alg:A1a} in a unified formula:
\begin{equation*}
\Delta{y}^k := -\nabla^2{d}_{\bar{\delta}}(y^k, t)^{-1}\nabla{d}_{\bar{\delta}}(y^k, t),
\end{equation*}
where $t$ can be $t_{k+1}$ or $t_0$.
We discuss in this subsection how to compute $\Delta{y}^k$ in an appropriate way by taking into account the specific structure of problem \eqref{eq:s1_main_CP}.
Note that $\Delta{y}^k$ is the solution of the following linear system:
\begin{equation}\label{eq:s7_linear_system}
\nabla^2{d}_{\bar{\delta}}(y^k, t)\Delta{y}^k = -\nabla{d}_{\bar{\delta}}(y^k, t).
\end{equation}
The gradient vector $\nabla{d}_{\bar{\delta}}(y^k,t)$ is computed as
\begin{equation*}
\nabla{d}_{\bar{\delta}}(y^k,t) = A\bar{x}_{\bar{\delta}}(y^k, t) - b = \sum_{i=1}^MA_{i}\bar{x}_{i}(y^k,t) - b := g_k,
\end{equation*}
and the Hessian matrix $\nabla^2{d}_{\bar{\delta}}(y^k,t)$ is obtained from
\begin{equation*}
\nabla^2{d}_{\bar{\delta}}(y^k,t) = \frac{1}{t}\sum_{i=1}^MA_{i}\nabla^2{F}_i(\bar{x}_{i}(y^k,t))^{-1}A_{i}^T := \sum_{i=1}^MA_iG_i^kA_i^T.
\end{equation*}
Note that each block $A_{i}\bar{x}_{i}(y^k,t)$ as well as $A_{i}\nabla^2{F}_i(\bar{x}_{i}(y^k,t))^{-1}A_{i}^T$ can be computed \textit{in parallel}.
Then, the linear system \eqref{eq:s7_linear_system} can be written as
\begin{equation}\label{eq:s7_NT_dual_dir}
\left(\sum_{i=1}^MA_iG_i^kA_i^T\right)\Delta{y}^k = -g_k.
\end{equation}
Sine matrix $G_i^k \succeq 0$ and $\sum_{i=1}^MA_iG_i^kA_i^T \succ 0$, one can apply either Cholesky-type factorizations  or conjugate gradient (CG) methods to
solve this problem. Note that the CG method only requires matrix-vector operations. More details on parallel solution of \eqref{eq:s7_linear_system} can be
found, e.g., in \cite{Necoara2009,Zhao2005}.

\section{Numerical Tests}\label{sec:num_results}
In this paper, we test the algorithms developed in the previous sections by solving a routing problem with congestion cost. This problem appears in the area of
telecommunications and in other network flow problems such as transportation \cite{Holmberg2006}. 
Let us consider a network $\mathcal{G} = (\mathcal{N}, \mathcal{A})$, where $\mathcal{N}$ is the set of nodes and $\mathcal{A}$ is the set of links. Let
$\mathcal{C}$ be a set of commodities to be sent through the network $\mathcal{G}$. Each commodity $k\in\mathcal{C}$ has a source $s_k\in\mathcal{N}$, a
destination $d_k\in\mathcal{N}$ and a certain amount of demand $d_k\geq 0$. 
Each link $(i,j)\in \mathcal{A}$ has a maximum capacity $b_{ij}\geq 0$ in which no congestion is assumed to be appeared, and a linear cost per unit $c_{ij}$.
The variable $u_{ijk}$ denotes the amount of commodity $k$ that is sent through the link $(i,j)$. Flow exceeding $b_{ij}$ may be sent through the link $(i,j)$
but will then causes congestion with an additional nonlinear cost function $g_{ij}$ depending on the exceeded value $v_{ij}$ considered as a variable.
We denote by $\mathcal{N}_s$ and $\mathcal{N}_d$ the sets of sources and destinations, respectively. Let $\mathcal{N}_c := \mathcal{N} \backslash
(\mathcal{N}_s\cup\mathcal{N}_d)$ and assume that each node in $\mathcal{N}_c$ has at least one ingoing link and one outgoing link.

Mathematically, the optimization model of the routing problem with congestion (RPC) can be formulated as, see, e.g. \cite{Holmberg2006}:
\begin{equation}\label{eq:RPC}
\left\{\begin{array}{cl}
\displaystyle\min_{u_{ijk}, v_{ij}} & \displaystyle\sum_{k\in \mathcal{C}}\displaystyle\sum_{(i,j)\in\mathcal{A}}c_{ij}u_{ijk} +
\displaystyle\sum_{(i,j)\in\mathcal{A}}w_{ij}g_{ij}(v_{ij}) \\
\textrm{s.t.} ~&  \displaystyle\sum_{j:(i,j)\in\mathcal{A}}u_{ijk} - \displaystyle\sum_{j:(j,i)\in\mathcal{A}}u_{jik} 
= \begin{cases}
d_k & \textrm{if} ~ i \in\mathcal{N}_s,\\
-d_k &\textrm{if}~  i \in\mathcal{N}_d, \\ 
0 &\textrm{otherwise},
\end{cases}\\ 
&\displaystyle\sum_{k\in C}u_{ijk} - v_{ij} = b_{ij}, ~(i,j)\in\mathcal{A},\\
&u_{ijk} \geq 0, ~ v_{ij} \geq 0, ~(i,j)\in\mathcal{A},\\
\end{array}
\right. 
\end{equation}
where $w_{ij} \geq 0$ is the weighting of the additional cost function $g_{ij}$ for $(i,j)\in\mathcal{A}$.

In this example we assume that the additional cost function $g_{ij}$ is given by one of the following functions: a) $g_{ij}(v_{ij}) = -\ln(v_{ij})$, the
logarithmic function or b) $g_{ij}(v_{ij}) = v_{ij}\ln(v_{ij})$, the entropy function.
With these choices, it was shown in \cite{Nesterov2004}, the self-concordant barrier function corresponding to the epi-graph  
\begin{equation*}
\mathcal{E}_{g_{ij}} := \left\{ (v_{ij}, s) \in \mathbb{R}_{+}\times\mathbb{R} ~|~ g_{ij}(v_{ij}) \leq s \right\}	
\end{equation*}
of $g_{ij}$ is given by: a) $F_{ij}(v_{ij}, s_{ij}) = -\ln v_{ij} - \ln(\ln v_{ij} + s_{ij})$ with parameter $\nu_{ij} = 2$ or b) $F_{ij}(v_{ij}, s_{ij}) = 
-\ln v_{ij} - \ln(s_{ij} - v_{ij}\ln v_{ij})$ with parameter $\nu_{ij} = 2$, respectively.
Now, by using slack variables $s_{ij}$, we can move the nonlinear terms of the objective function to the constraints. The objective function of the resulting
problem becomes
\begin{equation}\label{eq:RPC_objnew}
f(u, v, s) := \sum_{k\in \mathcal{C}}\displaystyle\sum_{(i,j)\in\mathcal{A}}c_{ij}u_{ijk} + \displaystyle\sum_{(i,j)\in\mathcal{A}}w_{ij}s_{ij}, 
\end{equation}
with additional constraints $g_{ij}(v_{ij}) \leq s_{ij}$, $(i,j)\in\mathcal{A}$.

It is clear that problem \eqref{eq:RPC} is separably convex with respect to $M$ components, $n$ variables $u_{ijk}$, $v_{ij}$ and $s_{ij}$ and $m$ coupling
constraints, where $M := n_{\mathcal{A}}$, $n := n_{\mathcal{C}}n_{\mathcal{A}} + 2n_{\mathcal{A}}$ and $m := n_{\mathcal{C}}n_{\mathcal{N}}$, where
$n_{\mathcal{A}} := \abs{\mathcal{A}}$, $n_{\mathcal{C}} := \abs{\mathcal{C}}$ and $n_{\mathcal{N}} := \abs{\mathcal{N}}$.
Let 
\begin{equation}\label{eq:RPC_constnew}
X_{ij} \!:=\! \left\{\! v_{ij} \! \geq \!0, \sum_{k\in\mathcal{C}}u_{ijk} \!-\! v_{ij} \!=\! b_{ij}, ~\! g_{ij}(v_{ij}) \!\leq\! s_{ij}, (i,j)\in \mathcal{A},
k\in\mathcal{C} \! \right\},
~(i,j)\in\mathcal{A}.
\end{equation}
Then problem \eqref{eq:RPC} can be reformulated in the form of \eqref{eq:s1_main_CP} with linear objective function \eqref{eq:RPC_objnew} and the local
constraint set \eqref{eq:RPC_constnew}. Note that each primal subproblem of the form \eqref{eq:s3_smoothed_dual_func} has $n_{\mathcal{C}} + 2$ variables and
one equality constraint.

The aim is to compare the effect of the parameters on the performance of the algorithms. We consider two variants of Algorithm \ref{alg:A1}, where we set
$\bar{\delta} = 0.5\bar{\delta}^{*}$ and $\bar{\delta} = 0.25\bar{\delta}^{*}$ in Phase 1 and $\bar{\delta} = 0.01$ and $\bar{\delta} = 0.005$ in Phase 2,
respectively. We denote these variants by \texttt{A1-v1} and \texttt{A1-v2}, respectively. For Algorithm \ref{alg:B1}, we also consider two
cases. In the first case we set the tolerance of the primal subproblem to $\varepsilon_p = 10^{-6}$, and the second one is $10^{-10}$, where we call them as
\texttt{A3-v1} and \texttt{A3-v2}, respectively. All variants are terminated with the same tolerance $\varepsilon_d = 10^{-4}$. The initial
barrier parameter value is set to $t_0 := 0.25$.

The algorithms are implemented in C++ running on a PC Desktop Intel\textregistered~Core(TM)2 Quad CPU Q6600 with 2.4GHz and 3Gb RAM. The algorithms are
paralellized by using \texttt{OpenMP}.
The input data is generated randomly, where the nodes of the network are generated in a rectangle $[0, 100]\times[0, 300]$, the demand $d_k$ is in
$[50, 500]$, the weighting vector $w$ is set to $10$, the congestion vector is in $[10, 100]$ and the linear cost $c_{ij}$ is the Euclidean length of the link
$(i, j)\in\mathcal{A}$. The nonlinear cost function $g_{ij}$ is chosen randomly between two functions in a) and b) defined above.

We test the algorithms on a collection of $150$ random problems, where $108$ problems are solve successfully. The size of these problems varies from $M = 6$
to $14.280$ components, $n = 84$ to $77.142$ variables and $m = 15$ to $500$ coupling constraints.

The performance profiles are shown in Figures \ref{fig:perp1} and \ref{fig:perp2}.
The first figure shows the performance profile of $4$ variants which consists of the total CPU time, the total time of solving the primal subproblems in two
phases, the CPU time of Phase 1 and the CPU time of Phase 2 separately in second.
\begin{figure}[ht!]
\centerline{\includegraphics[angle=0,height=7.0cm,width=13.1cm]{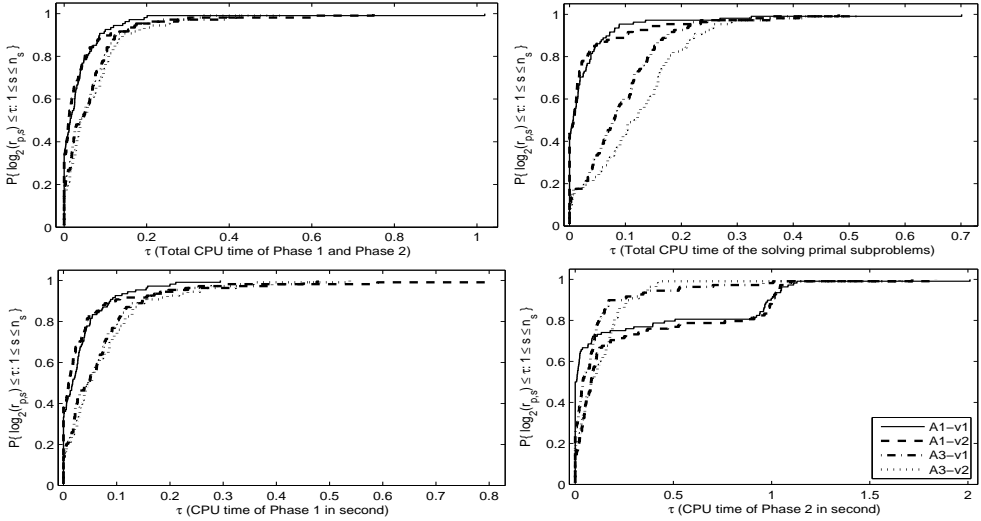}}
\caption{The CPU time performance profile of four variants.}
\label{fig:perp1}
\end{figure}
As we can see from this figure that Algorithm \ref{alg:A1} works better than Algorithm \ref{alg:B1} in terms of the total CPU time and the CPU time for solving
the primal subproblems. Moreover, the accuracy in solving the primal subproblems also affects the performance of the algorithms.
We also observe that the number of iterations for solving the master problem in Phase 1  for all four variants are almost similar, while
they are different in Phase 2. However, Phase 2 is performed when the iteration point is in the quadratic convergence region, it only takes few steps toward 
the desired approximate solution. Therefore, the computational time of Phase 1 dominates the one in Phase 2.
Moreover, in this example, the structure of the master problem is almost dense, we do not use any sparse linear algebra solver. 
Consequently, the algorithms developed in this paper are recommended to the class of problems with many variables and few coupling constraints in the case
the master dual problem possesses dense structure.
In other applications, the efficient methods for sparse linear algebra should be taken into account.

We also compare the total number of iterations for solving the primal subproblems in Figure \ref{fig:perp2}. It can be seen from this figure that Algorithm
\ref{alg:A1} is superior in terms of iterations compared to Algorithm \ref{alg:B1}, although the accuracy of solving the primal subproblem in Algorithm
\ref{alg:B1} is set to $10^{-6}$, which is not too high in interior point methods.
\begin{figure}[ht!]
\centerline{\includegraphics[angle=0,height=3.9cm,width=7.6cm]{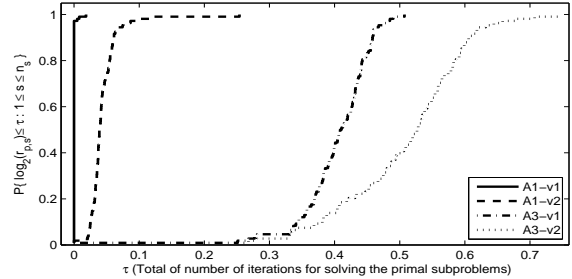}}
\caption{The iteration performance profile of four variants.}
\label{fig:perp2}
\end{figure}
The performance profiles also reveal the effect of the parameters on the number of iterations and computational time.
Consequently, in practice, it is valuable to carefully choose appropriate parameters for a specific implementation.

\section{Concluding remarks}\label{sec:conclusions}
We have proposed a smoothing technique for Lagrangian decomposition using self-concordant barriers in large-scale convex separable optimization. We provided
global and local approximations to the dual function. Then, we proposed a path-following algorithm with inexact perturbed Newton iterations. The convergence of
the algorithm has been analyzed and its complexity has been estimated. The theory presented in this paper is significant in practice, since it allows to solve
the primal subproblem inexactly. Moreover, we allow one to balance between the accuracy of solving the primal subproblem and the convergence rate of the
path-following algorithm.
Even in the exact case, we also obtained a direct analysis for the convergence of the path-following algorithm which was presented by Mehrotra
\cite{Mehrotra2009} \textit{et al} and Shida \cite{Shida2008}. The details of implementation and numerical tests have also been presented.
Extensions to the inexactness of linear algebra and to distributed implementation are an interesting and significant future research direction.

\vskip 0.2cm
\begin{small}
\noindent{\textbf{Acknowledgments.}} This research was supported by Research Council KUL: CoE EF/05/006 Optimization in Engineering(OPTEC), GOA AMBioRICS,
IOF-SCORES4CHEM, several PhD/postdoc \& fellow grants; the Flemish Government via FWO: PhD/postdoc grants, projects G.0452.04, G.0499.04, G.0211.05,
G.0226.06, G.0321.06, G.0302.07, G.0320.08 (convex MPC), G.0558.08 (Robust MHE), G.0557.08, G.0588.09, research communities (ICCoS, ANMMM, MLDM) and via IWT:
PhD Grants, McKnow-E, Eureka-Flite+EU: ERNSI; FP7-HD-MPC (Collaborative Project STREP-grantnr. 223854), Contract Research: AMINAL, HIGHWIND, and Helmholtz
Gemeinschaft: viCERP; Austria: ACCM, and the Belgian Federal Science Policy Office: IUAP P6/04 (DYSCO, Dynamical systems, control and optimization, 2007-2011), 
European Union FP7-EMBOCON under grant agreement no 248940; CNCS-UEFISCDI (project TE231, no. 19/11.08.2010); ANCS (project PN II, no. 80EU/2010); Sectoral
Operational Programme Human Resources Development 2007-2013 of the Romanian Ministry of Labor, Family and Social Protection through the Financial Agreement
POSDRU/89/1.5/S/62557.
\end{small}

\appendix

\section{The proof of the technical statements} In this appendix, we provide a complete proof of Lemmas \ref{le:s3_local_bound_on_dual_func}, 
\ref{le:s3_bound_of_smoothed_dual_func_to_orig_dual_func} and \ref{le:s3_choice_of_para_t}.

\subsection{The proof of Lemma \ref{le:s3_local_bound_on_dual_func}}
\begin{proof}
Since $F_i$ is standard self-concordant, according to \cite[Theorem 4.1.7, inequality 4.1.8]{Nesterov2004} we have
\begin{align*}
F_i(y_{i}) &\geq F_i(x_{i}) + \nabla F_i(x_{i})^T(y_{i}-x_{i}) + \omega(\norm{y_{i}-x_{i}}_{x_{i}}) \\
& \geq F_i(x_{i}) -\norm{\nabla F_i(x_{i})}^{*}_{x_{i}}\norm{y_{i}-x_{i}}_{x_{i}} + \omega(\norm{y_{i}-x_{i}}_{x_{i}}).
\end{align*}
This inequality implies
\begin{align*}
F_i(x_{i}) - F_i(y_{i}) &\leq \norm{\nabla F_i(x_{i})}^{*}_{x_{i}}\norm{y_{i}-x_{i}}_{x_{i}} - \omega(\norm{y_{i}-x_{i}}_{x_{i}}) \\
& \leq \max_{x_{i}\in\dom(F_i)}\left\{ \norm{\nabla F_i(x_{i})}^{*}_{x_{i}}\norm{y_{i}-x_{i}}_{x_{i}} - \omega(\norm{y_{i}-x_{i}}_{x_{i}}) \right\}\\
& \leq \max_{\xi := \norm{y_{i}-x_{i}}_{x_{i}} \geq 0 }\left\{ \norm{\nabla F_i(x_{i})}^{*}_{x_{i}}\xi - \omega(\xi) \right\}\\
& = \omega_{*}(\norm{\nabla F_i(x_{i})}_{x_{i}}^{*}).
\end{align*}
Here, the last equality follows from \cite[Lemma 4.1.4]{Nesterov2004} and the assumption that $\lambda_{F_i}(x^{*}_{i}(y,t)) < 1$.
Using the above inequality with $y_i = x_i^c$ and $x_i = x_i^{*}(y,t)$ we have
\begin{align}\label{eq:lm21_est0}
F_i(x_{i}^{*}(y,t)) - F_i(x_{i}^c) &\leq \omega_{*}(\lambda_{F_i}(x^{*}_{i}(y,t)).
\end{align}
Now, we prove \eqref{eq:s3_local_bound_on_dual_func}. Let $x_{i}(\alpha) := x^{*}_{i}(y, t) + \alpha(x^{*}_{i}(y) - x^{*}_{i}(y,t))$ with $\alpha \in [0,1)$.
Since $x^{*}_{i}(y,t)\in\int(X_{i})$ and $\alpha < 1$, $x_{i}(\alpha) \in\int(X_{i})$.
By applying \cite[inequality 2.3.3]{Nesterov1994}, we have
\begin{equation*}
F_i(x_{i}(\alpha)) \leq F_i(x^{*}_{i}(y,t)) - \nu_i\ln(1-\alpha),
\end{equation*}
which is equivalent to
\begin{equation}\label{eq:lm21_est1}
F_i(x_{i}(\alpha)) - F_i(x^c_{i})\leq F_i(x^{*}_{i}(y,t)) - F_i(x^c_{i}) - \nu_i\ln(1-\alpha).
\end{equation}
Now, from the definition of $d_i(y,t)$, the concavity of $\phi_i$ and $d_i(y)$, and \eqref{eq:lm21_est0} we have
\begin{eqnarray}\label{eq:lm21_est2}
d_i(y,t) && = \max_{x_{i}\in\int(X_{i})}\left\{\phi_i(x_{i}) + y^TA_{i}x_{i} - t[F_i(x_{i}) - F_i(x_{i}^c)]\right\} \nonumber\\
&&\geq \max_{\alpha\in [0,1)}\left\{\phi_i(x_{i}(\alpha)) + y^TA_{i}x_{i}(\alpha) - t[F_i(x_{i}(\alpha)) - F_i(x_{i}^c)]\right\} \nonumber\\
&&\geq \max_{\alpha\in [0,1)}\Big\{\alpha[\phi_i(x^{*}_{i}(y)) + y^TA_{i}x^{*}_{i}(y)] + (1-\alpha)[\phi_i(x^{*}_{i}(y,t)) + y^TA_{i}x^{*}_{i}(y,t)]\nonumber\\
[-1.5ex]\\[-1.5ex]
&& - t[F_i(x_{i}(y,t)) - F_i(x_{i}^c)] + \nu_it\ln(1-\alpha)\Big\} \nonumber\\
&& = \max_{\alpha\in [0, 1)}\Big\{\alpha d_i(y) + (1-\alpha)d_i(y,t) - \alpha t[F_i(x_{i}(y,t)) - F_i(x_{i}^c)] + \nu_it\ln(1-\alpha) \Big\} \nonumber\\
&& \overset{\tiny\eqref{eq:lm21_est0}}{\geq} \max_{\alpha\in [0, 1)}\Big\{ \alpha d_i(y) + (1-\alpha)d_i(y,t) + t\nu_i\ln(1-\alpha) - \alpha
t\omega_{*}(\lambda_{F_i}(x^{*}_{i}(y,t)))\Big\}.\nonumber
\end{eqnarray}
Rearranging \eqref{eq:lm21_est2}, we obtain
\begin{equation}\label{eq:lm21_est3}
d_i(y, t) \geq d_i(y) -t\omega_{*}(\lambda_{F_i}(x^{*}_{i}(y,t))) + t\nu_i\frac{\ln(1-\alpha)}{\alpha}, ~\forall \alpha\in [0, 1).
\end{equation}
Since $\frac{\ln(1-\alpha)}{\alpha} \leq -1$ for all $\alpha \in (0, 1)$ and $\lim_{\alpha\to 0^{+}}\frac{\ln(1-\alpha)}{\alpha} = -1$. Inequality
\eqref{eq:lm21_est3} implies that
\begin{equation*}\label{eq:lm21_est4}
d_i(y, t) - d_i(y) \geq  -t[\omega_{*}(\lambda_{F_i}(x^{*}_{i}(y,t))) + \nu_i].
\end{equation*}
which is the right-hand side of \eqref{eq:s3_local_bound_on_dual_func}. The left-hand side of \eqref{eq:s3_local_bound_on_dual_func} follows from the relation
$F_i(x_{i}) - F_i(x_{i}^c) \geq \omega(\norm{x_i -
x^c_{i}}_{x^c_{i}}) \geq 0$ due to \eqref{eq:s3_analytic_center_est}.
\end{proof}

\subsection{The proof of Lemma \ref{le:s3_bound_of_smoothed_dual_func_to_orig_dual_func}}
\begin{proof}
The second inequality in \eqref{eq:s3_bound_of_smoothed_dual_func_to_orig_dual_func} is proved in Lemma \ref{le:s3_local_bound_on_dual_func}.
We now prove the third one.
Let us denote by $x_{i}^{\tau}(y) := x_{i}^c + \tau(x^{*}_{i}(y) - x^c_{i})$, where $\tau\in [0,1]$.
Since $F_i$ is $\nu_i$-self-concordant, it follows from \cite[inequality (2.3.3)]{Nesterov1994} that
\begin{equation*}\label{eq:s3_proof1_est1}
F_i(x_{i}^{\tau}(y)) \leq F_i(x^c_{i}) - \nu_i\ln(1-\tau), ~\tau \in [0, 1).
\end{equation*}
Combining this inequality and the concavity of $\phi_i$ we have
\begin{eqnarray}\label{eq:s3_proof1_est2}
d_i(y, t) &&= \max_{x_{i}\in\textrm{int}(X_{i})}\left\{ \phi_i(x_{i}) + y^TA_{i}x_{i} - t[F_i(x_{i}) - F_i(x_{i}^c)]\right\}\nonumber\\
&&\geq \max_{\tau\in [0,1)}\left\{ \phi_i(x^{\tau}_{i}(y)) + y^TA_{i}x_{i}^{\tau}(y) - t[F_i(x^{\tau}_{i}(y)) - F_i(x_{i}^c)] \right\}\nonumber\\
[-1.5ex]\\[-1.5ex]
&&\geq \max_{\tau\in [0,1)}\Big\{ (1 \!-\! \tau)[\phi_i(x_{i}^c) \!+\! y^TA_{i}x_{i}^c] \!+\! \tau[\phi_i(x^{*}_{i}(y)) \!+\! y^TA_{i}x_{i}^{*}(y)]  \!+\!
t\nu_1\ln(1 \!-\! \tau)\Big\} \nonumber\\
&& = \max_{\tau\in [0,1)}\left\{ (1-\tau)d_i^c(y) + \tau d_i(y) + t\nu_i\ln(1-\tau)\right\}. \nonumber
\end{eqnarray}
Now, we maximize the function $\xi(\tau) := (1-\tau)d_i^c(y) + \tau d_i(y) + t\nu_i\ln(1-\tau)$ in last line of \eqref{eq:s3_proof1_est2} with respect to
$\tau\in [0, 1)$ to obtain $\tau^{*} = \left[1 -  \frac{t\nu_i}{d_i(y) - d_i^c(y)}\right]_{+}$, where $[a]_{+} := \max\{0,a\}$.
Therefore, if $\frac{d_i(y) - d_i^c(y)}{t\nu_i}\leq 1$, i.e. $\tau^{*} = 0$, then $d_i(y) - d_i^c(y) \leq t\nu_i$. Otherwise, by substituting $\tau^{*}$ into
the last line of
\eqref{eq:s3_proof1_est2}, we obtain
\begin{equation*}\label{eq:s3_proof1_est3}
d_i(y) \leq d_i(y,t) + t\nu_i\left(1 + \left[\ln\frac{d_i(y) - d_i^c(y)}{t\nu_i}\right]_{+}\right).
\end{equation*}
Summing up this inequality for $i=1,2$ we get \eqref{eq:s3_bound_of_smoothed_dual_func_to_orig_dual_func}.
\end{proof}

\subsection{The proof of Lemma \ref{le:s3_choice_of_para_t}}
\begin{proof}
Let us fix $\kappa \in (0, 1)$, it is trivial that  $\ln(x^{-1}) \leq x^{-\kappa}$ for $0 < x \leq \kappa^{1/\kappa}$.
Therefore, we have
\begin{equation*}
\nu_it(1+[\ln(K_i/(\nu_it)]_{+}) \leq \nu_it(1 +
(K_i/(\nu_it))^{-\kappa}) \leq [\nu_i + \nu_i(K_i/\nu_i)^{\kappa}]t^{1-\kappa}, ~~ \forall t \leq \frac{\nu_i}{K_i}\kappa^{1/\kappa}.
\end{equation*}
Consequently, if $ t \leq \min\big\{\frac{\nu_i}{K_i}\kappa^{1/\kappa}, \big(\frac{\varepsilon}{2[\nu_i + \nu_i(K_i/\nu_i)^{\kappa}}\big)^{1/(1-\kappa)}\big\}$ 
then
$\nu_it(1+[\ln(K_i/(\nu_it)]_{+}) \leq 0.5\varepsilon$.
Summing up this inequality for $i=1,2$, we get \eqref{eq:s3_choice_of_t} in Lemma \ref{le:s3_choice_of_para_t}.
\end{proof}

\bibliographystyle{plain}

\end{document}